\documentclass[11pt,,fleqn]{amsart} 

\usepackage[usenames,dvipsnames,condensed]{xcolor}

\usepackage{amsthm}
\usepackage{amsfonts}
\usepackage[english]{babel}
\usepackage[usenames]{xcolor}
\usepackage{graphicx}
\usepackage{soul}
\usepackage{stfloats}
\usepackage{morefloats}
\usepackage{cite}
\usepackage{lscape}
\usepackage{epstopdf}
\usepackage{braket}
\usepackage[lite]{amsrefs}

\usepackage{amsmath,amstext,amsopn,amsfonts,eucal,amssymb}
\usepackage{graphicx,wrapfig,url}

\newcommand\Z{{\mathbb Z}}

\newtheorem{theorem}{Theorem}[section]
\newtheorem{corollary}[theorem]{Corollary}
\newtheorem{lemma}[theorem]{Lemma}
\newtheorem{proposition}[theorem]{Proposition}
\newtheorem{definition}[theorem]{Definition}

\newtheorem{example}[theorem]{Example}

\newtheorem{remark}[theorem]{Remark}

\newcommand{\cal}{\mathcal}

\begin{document}

\title{Fundamental Heap for  Framed Links and Ribbon Cocycle Invariants}

\author{Masahico Saito} 
\address{Department of Mathematics, 
	University of South Florida, Tampa, FL 33620, U.S.A.} 
\email{saito@usf.edu} 

\author{Emanuele Zappala} 
\address{Department of Computational Medicine \& Bioinformatics, 
	University of Michigan, Ann Arbor, MI 48109-2800, U.S.A.} 
\email{zae@usf.edu, zemanuel@umich.edu}

\begin{abstract}
A heap is a set with a 
	certain ternary operation that is self-distributive (TSD) and
exemplified by a group with 
the operation $(x,y,z)\mapsto xy^{-1}z$. 
We introduce and investigate  framed link invariants using heaps.
 In analogy with the knot group, we  define the fundamental heap of framed links using group presentations.
The fundamental heap is determined for some classes of links such as certain families of torus and pretzel links. We show that for these families of links there exist epimorphisms from  fundamental heaps to Vinberg and Coxeter groups, implying that  corresponding groups are infinite. 
A relation to the Wirtinger presentation is also described. 

The cocycle invariant 
is defined using ternary self-distributive (TSD) cohomology, by means of a state sum that uses ternary heap $2$-cocycles as weights. 
This invariant corresponds to a rack cocycle invariant for the rack constructed by
doubling of a heap, while colorings can be regarded as heap morphisms from the fundamental heap.

For the construction of the invariant, first computational methods for the heap cohomology are developed.
It is shown that the cohomology  splits into two types, called degenerate and nondegenerate, and that the degenerate part is one dimensional.
Subcomplexes are constructed based on group cosets, that allow computations of the nondegenerate part. 
Computations of the cocycle invariants are presented using the cocycles constructed, and 
conversely, it is proved that the invariant values can be used to  derive algebraic properties of the cohomology.
\end{abstract}

\maketitle

\date{\empty}

\tableofcontents

\section{Introduction}

	A heap is a set with a ternary operation satisfying two properties called para-associativity and degeneracy condition. The fundamental example of heap is a group $G$ with operation $(x,y,z) \mapsto xy^{-1}z$ (which we call a {\it group heap}). Moreover, it can be proved that the category of pointed heaps (i.e. heaps with an arbitrary choice of a fixed element) and the category of group heaps are 
	equivalent \cite{BH}. 
		It has been observed that heaps do satisfy 
		the {\it ternary self-distributive} (TSD) property, along with 
		a few other properties~\cite{ESZheap}.
		A ternary operation $(x,y,z)\mapsto T(x,y,z)$ is called ternary self-distributive if it satisfies
	$$T ( T(x, y, z ), u, v)= T ( T(x, u, v ),  T(y, u, v ),  T(z, u, v ) ) $$
	for all $x, y, z, u, v$. 
	This	makes a heap a ternary analogue of a {\it quandle}. 
		Quandles  
	  are binary self-distributive structures 
	 that have been  extensively studied  for constructing knot invariants in recent decades. 
Diagrammatic representations using link diagrams for higher-arity self-distributive  operations have been also used as computational aids in \cite{ESZ}, and their cohomology has been studied. 
Knot invariants using ternary operations have been studied, for example, in 
	  \cite{NOO,Nie1,Nie2}, in which colorings are assigned to complementary regions,
	  while in this paper, colorings are assigned on doubled arcs of framed links.
	  
The purpose of this paper is to introduce and study framed link invariants, named 
	the {\it fundamental heap} and {\it ribbon cocycle invariants}, that utilize 
		 heaps  and their ternary self-distributive cohomology. 
	The fundamental heap is defined by group presentations in a  manner similar  to the knot group being defined through Wirtinger presentation. 
	We construct the cocycle  invariants using 
	a state-sum, weighted by heap $2$-cocycles in analogy to  quandle cocycle invariants
	\cite{CJKLS}. 
	The main novelty is that  colorings by group heaps  for the doubled arcs
	are used in a manner essential to ternary operations beyond composing binary operations, and using groups allows us to define  the fundamental heap by group presentations.
The existence part of the ternary cocycle invariant has also been included in the PhD dissertation of the second author \cite{EZ}, where the construction is given for general ternary self-distributive operations and their cohomology. Here we focus our attention on the class of heap invariants. 
		
\begin{figure}[htb]
\begin{center}
\includegraphics[width=3.5in]{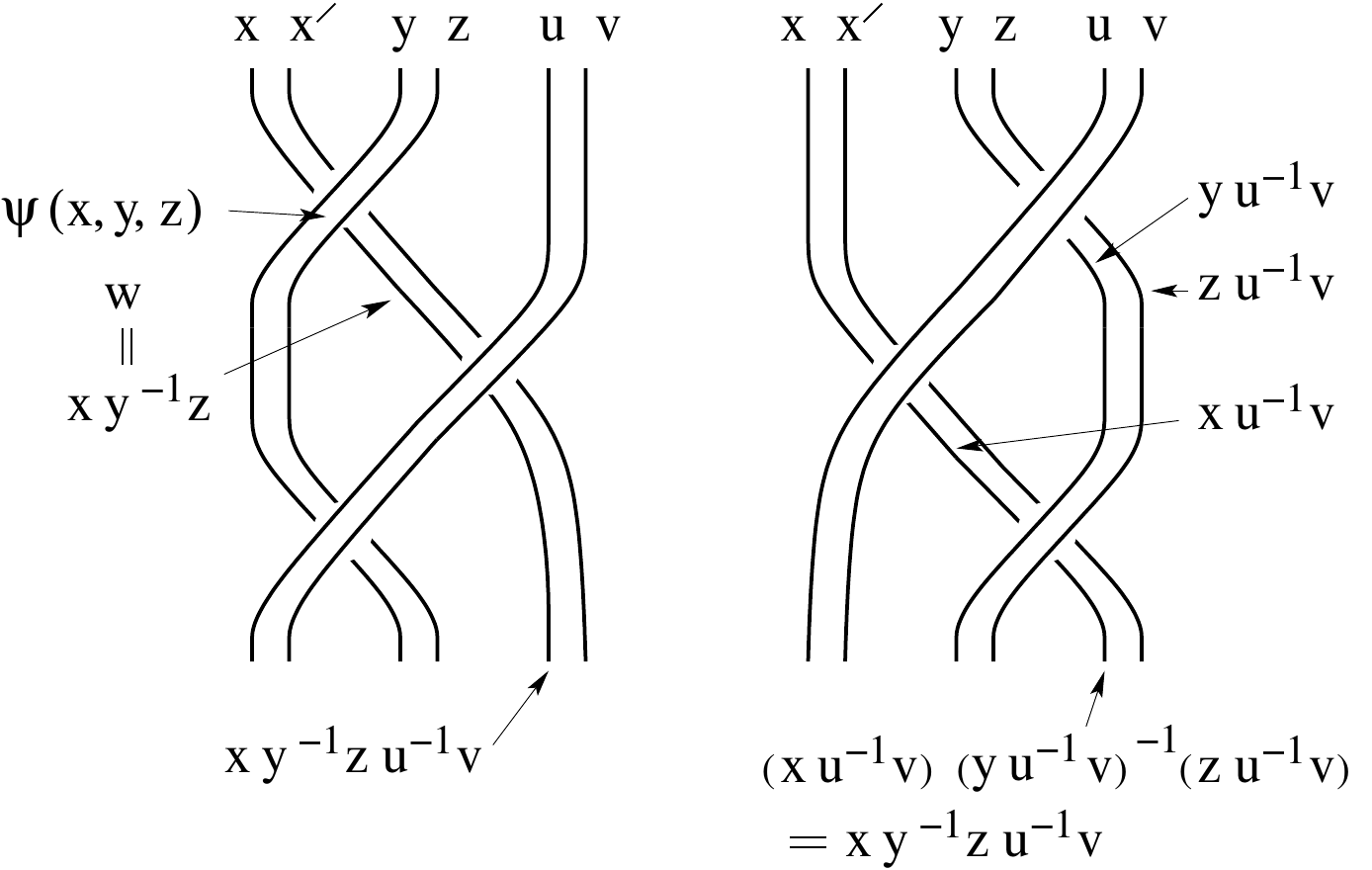}
\end{center}
\caption{Heap and Reidemeister type III move}
\label{heaptypeIII}
\end{figure}

More specifically, a coloring of a  framed link by a group heap $X$ is defined as indicated in Figure~\ref{heaptypeIII}.
A blackboard framed link diagram is thickened to a ribbon diagram with doubled arcs. Each of the doubled arcs is colored by a group element. 
In the top left  of the figure, a doubled crossing is colored by $(x, x')$ and $(y,z)$ at the top arcs.
The left arc below the crossing traced from the arc colored $x$ is colored $w$, which is required 
to be equal to $xy^{-1}z$, the value of heap operation. The other arc is similarly colored by $x' y^{-1}z$. 
At the bottom right of $x$-colored string, the left- and right-hand side of the Reidemeister type III move in Figure~\ref{heaptypeIII}, the implications of the coloring condition are computed. The output
corresponds to 
the ternary self-distributivity, as expected, for the heap operation.
Here the use of a ternary operation is essential; there is no arc colored by an element  obtained only from $x$ and $y$. The simplicity of the expression $xy^{-1}z$ for group elements $x, y, z$  is also notable.

The cocycle invariant is constructed from two 
	procedures. That of taking all possible heap colorings of a diagram, and that of forming, for each given coloring, the product of (Boltzmann) weights over all  crossings.
	The summation is 
	then 	taken over all colorings.
		The  weight  is given by $2$-cocycles of the TSD  cohomology of heaps 
	evaluated for 
	 the elements coloring the four out of eight 
	 (doubled) arcs meeting at a crossing.
 It is crucial therefore to compute TSD cohomology of heaps, or at least to produce large nontrivial classes in the second cohomology group, 
 for a given heap $X$ and an abelian coefficient group $A$.  To this objective we compute the second cohomology group of some heaps and provide 
	examples of nontrivial classes. Moreover, we prove that TSD (co)homology of heaps is obtained as direct sum of two types of (co)homology that we call {\it degenerate} and {\it nondegenerate}, in analogy with the notion
	for the binary case. We prove that degenerate cohomology is in fact 
	generated by a single element, 
	for any heap $X$ and abelian coefficient group $A$, 
	a unique feature to heap cohomology comparing to 
	degenerate cohomology of quandles.
	In order to compute the nondegenerate part we introduce certain subcomplexes, named {\it coset} subcomplexes based on group left cosets, and prove that these fit in long exact sequences. Moreover we prove that the procedure of taking coset subcomplexes and corresponding long exact sequences can be iterated, 
	which is useful in computations of cohomology groups.

	To obtain  colorings and motivated from Wirtinger presentation of knot groups, we  introduce the  invariant of framed links named {\it fundamental heap}, obtained 
		from a link diagram with blackboard framing.
		Generators are assigned to parallel arcs of framed link diagrams, 
	and
the heap relation 
is assigned at each crossing of the diagram. The heap of the group so defined is an invariant of the framed link. We prove that a fundamental heap is  in general given by the free product of a free group 
of the rank equal to the number of components, and another  group.
 For some families of examples, namely some torus and pretzel links, we prove that the non-free part of the fundamental heap is infinite and non-abelian as well, and we prove this by explicitly constructing  epimorphisms to Vinberg groups. In some specific cases we show that it is possible to take the Vinberg groups to be Coxeter groups. 
 Relations to Wirtinger presentation of link groups are presented. 
	
	We apply the computations of cohomology groups and nontrivial $2$-cocycles,
	together with colorings obtained from the fundamental heap, 
	to  compute nontrivial cocycle invariants for families of framed links.
	Moreover, we show that nontriviality of cocycle invariants can be used to derive cohomological properties of the heap used. 
	 In fact, we provide
	lower bounds for the rank of the second cohomology group of dihedral heaps, i.e. heaps associated to dihedral groups, of arbitrary order.
	
	The cocycle invariant defined here can be regarded as the rack cocycle invariant \cite{CJKLS}
	for the rack constructed by doubling the TSD operations as
	$(x,y)*(u,v):=(T(x,u,v), T(y,u,v))$ defined on $X^2$ for a TSD set $X$ and $x,y,u,v \in X$.
	These constructions of new self-distributive sets from old in cascades have been 
	studied in \cite{ESZ}.
	The cascade constructions  in \cite{ESZ}, on the other hand, were motivated from the double string diagrams
	for heap colorings. 
		This novel interpretation via doubled strands of the cocycle invariant has twofold applications. Firstly, it allows to use geometric arguments regarding the nontriviality of families of framed links, to derive algebraic results on the second cohomology groups of heaps. Secondly, we systematically obtain corresponding nontriviality results on the cohomology of doubled racks. Therefore, computations of cohomology specifically applicable to heaps are presented here, and  further developments in using heaps for topological invariants for compact orientable surfaces with boundary embedded in space have been made \cite{SZsfceribbon} as well, where the doubled strand interpretation plays a fundamental role in the definition of the cocycle invariant, which has no clear rack-theoretic analogue. 
	Furthermore, the use of heaps for doubled strands specifically motivates the definition of the fundamental heap, which is of interest on its own, and is one of the main results of this paper.
	Thus we  focus on and specialize to heaps in formulating the invariants throughout the paper,
	instead of reverting to the rack formulation. 

	The organization of the paper is as follows. In Section~\ref{sec:Pre} we give an overview of the basic definitions of heap, ternary self-distributive (TSD) (co)homology and blackboard framings of diagrams representing framed links. In Section~\ref{sec:coh}, the ternary self-distributive cohomology of heaps 
	 is studied. In Section~\ref{sec:Nontrivial} we compute the second cohomology group of some cyclic group heaps, thereby providing some first examples of nontrivial TSD $2$-cocycles of heaps. In Section~\ref{sec:degnondeg} we define two subcomplexes, named degenerate and nondegenerate, and prove that TSD cohomology of heaps splits in the direct sum of degenerate part and nondegenerate part. We prove that the degenerate part is one dimensional and provide some constructive examples of nontriviality of nondegenerate cohomology. 
	 Cocycles are constructed for $\Z_n$ and $D_n$ in Section~\ref{sec:cocy}. 
	 In Section~\ref{sec:coset} we define subcomplexes that simplify computations of nondegenerate cohomology and apply our methods to give examples of computations 
	for 
	dihedral group heaps $D_{n}$.
		 In Section~\ref{sec:fund} we define what we call  the fundamental heap of a framed link and compute it for some family of examples, for which we prove it is infinite by constructing an epimorphism to infinite Vinberg groups. We prove some important properties of the fundamental heap, such as the fact that it is obtained as a free product.
	A relation with the Wirtinger presentation of the knot  group is also presented. In Section~\ref{sec:color} we define heap colorings of ribbon diagrams of framed links and cocycle invariants. 
	We provide several examples of nontrivial invariants, using the computations of Section~\ref{sec:coset}. Moreover we prove a lower bound on the rank of second nondegenerate cohomology group of dihedral group heaps, by means of the cocycle invariant of torus links $T(2,2n)$. 
	Some computations are 
	deferred to the Appendices.

\section{Preliminary}\label{sec:Pre}

In this section we review materials used in this paper. 

\subsection{Heaps}

In this section we recall the definition and basic properties of heaps.
Given a set $X$ with a ternary operation $[ - ]$, the set of equalities 
$$ 
 [  [ x_1, x_2, x_3 ], x_4, x_5  ]=  [  x_1,[  x_4, x_3, x_2 ]  ,  x_5 ] = [  x_1,x_2, [  x_3, x_4, x_5 ]  ]   
$$ 
 is called para-associativity.
The 
equations $[x,x,y]=y$ and $[x, y, y ] = x$ are called the degeneracy conditions.
A {\it heap} is a non-empty set with a ternary operation satisfying 
		the para-associativity  
		and the degeneracy conditions~\cite{ESZheap}.

A typical example of a heap is a group $G$ where the ternary operation is given by $[x,y,z]=xy^{-1}z$,
which we call a {\it group heap}. If $G$ is abelian, we call it an {\it abelian (group) heap}. 
Conversely, 
given a heap $X$ with a fixed element $e$,  
one defines a binary operation on $X$ by $x*y=[x,e,y]$ which makes $(X,*)$ into a group with $e$ as the identity,  and the inverse of $x$ is $[e,x,e]$ for any $x \in X$.  Moreover, the associated group heap coincides with the initial heap structure. Focusing on group heaps is therefore not a strong restriction, as it is always possible to construct a group whose group heap coincides with an arbitrary heap. 
Although the definition of heap is derived from that of para-associativity, which is a ternary generalization of associativity, our focus is on the ternary self-distribitivity 
as described below that heaps satisfy. 

Let $X$ be a set   with a ternary operation  $(x,y,z)\mapsto  T(x,y,z)$.
The condition 
$$ T ( ( x, y,z ) ,  u,v)= T ( T(x,u,v), T(y,u,v) T(z,u,v ) )$$
 for all $x,y,z,u,v \in X$, is called {\it  ternary 
	 self-distributivity}, TSD for short, 
	 and a set with a TSD operation is called {\it ternary self-distributive set}, or TSD set for short. 
It is known and easily checked that the heap operation $(x,y,z)\mapsto  [x,y,z]=T(x,y,z)$ is ternary self-distributive.

\subsection{Ternary self-distributive homology}

The ternary self-distributive (co)homology was studied in \cite{EGM,Green,ESZ} which we review.
Let $X$ be a ternary self-distributive set. The $n$-dimensional chain group $C_n^{\rm SD}(X)$ 
is the free abelian group generated by $(2n - 1)$-tuples $(x_1, x_2, \ldots, x_{2n-1})$.
The boundary operator $d_n: C_n^{\rm SD} (X) \rightarrow  C_{n-1}^{\rm SD} (X)$ 
is defined by 
\begin{eqnarray*}
\lefteqn{ d_n (x_1, x_2, \ldots, x_{2n-1}) =}\\
& &  \sum_{i=1}^{n-1} (-1)^i [\ (x_1, \ldots, \widehat{x_{2i},x_{2i+1}},\ldots, x_{2n-1}) \\
& & \hspace{15mm}
- (x_1 x_{2i}^{-1}x_{2i+1},  \ldots, x_{2i-1} x_{2i}^{-1} x_{2i+1} , \widehat{x_{2i},x_{2i+1}}, x_{2i+2}, \ldots  , x_{2n-1}) \ ] .
\end{eqnarray*}
Cycle, boundary, homology groups are as usual denoted by $Z_n^{\rm SD} (X)$,
$B_n^{\rm SD} (X)$, and  $H_n^{\rm SD} (X)$, respectively, and similar for the cochain, cocycle, coboundary,  and cohomology groups.  Throughout this article we will use the convention that the $n^{\rm th}$ cohomological differential $\delta^n$ is obtained by dualizing the $(n+1)^{\rm st}$ homological differential. Therefore, $2$-cocycles are defined as those $2$-cochains that are in the kernel of $\delta^2$, which is obtained by duallizing $d_3$, 
so that $(\delta^2 \psi) (x_1, \ldots, x_5)= \psi (d_3 (x_1, \ldots, x_5))$ for 
$\psi \in C^3(X)$.

\begin{remark}
	{\rm 
	We point out that the differentials for the TSD (co)homology delete elements in pairs, starting from the pair $(x_2,x_3)$. It is therefore useful to think of the $(2n-1)$-tuples as being grouped as $(x_1,(x_2,x_3), \ldots , (x_{2n - 2},x_{2n-1}))$. As it will be seen below, in Section~\ref{sec:color}, this is also reflected geometrically in the definition of the cocycle invariant. This notation was used in \cite{ESZ}, where the diagrammatic interpretation of the above differentials was given in terms of curtain diagrams. Similar considerations appear also in \cite{NOO}, where a homology theory is defined, where all the entries of the defining chains are paired. We will not explicitly make use of these notations in this article, for simplicity. 
}
\end{remark}

\subsection{Blackboard framing of  link diagrams}

We use the blackboard framing of a diagram to represent a framed link that is commonly used. 
Specifically, a link diagram is thickened to a ribbon to represent a given framing, and thickening is performed on the plane for a given  projection, so that parallel strings of the boundary of a ribbon does not appear twisted on projection or diagrams. We call this thickened diagram a {\it ribbon} diagram.
In Figure~\ref{crossing} (A), a diagram at a crossing is depicted. Labels will be used in later sections.
In (B), a ribbon diagram with double (parallel) arcs is depicted, that represent the framing by the convention of the blackboard framing.

\begin{figure}[htb]
\begin{center}
\includegraphics[width=2.6in]{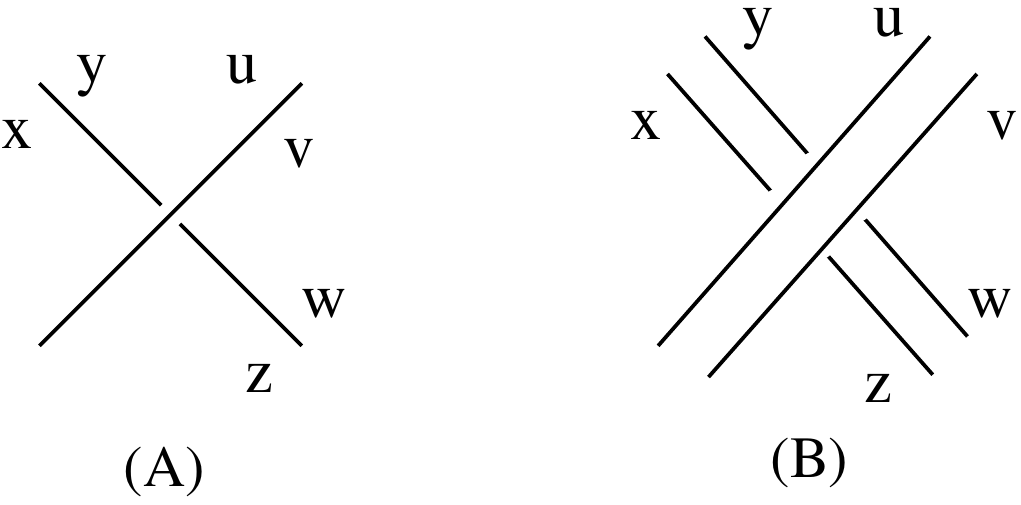}
\end{center}
\caption{Colors at a crossing}
\label{crossing}
\end{figure}

Full twists for framed links are realized by self intersections. For example, Figure~\ref{telecord} represents positive twists being added to two parallel strings representing a trivial arc. Labels in the figure will be used in later sections.
Either orientation of the arc gives rise to a number of positive twists, so that we call them  positive twists, or {\it telephone cord with writhe $n$}, denoted by $C_n$. 
Negative twists are represented by the opposite crossing information at every crossing.
A positive crossing 
is defined to be the crossing in Figure~\ref{crossing} (A) 
with the orientations of both arcs pointing downward.

\begin{figure}[htb]
\begin{center}
\includegraphics[width=3.5in]{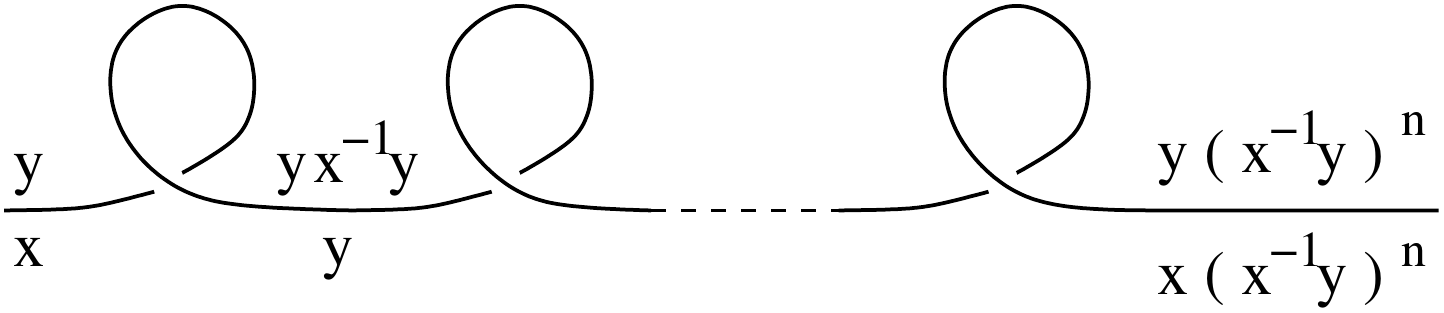}
\end{center}
\caption{Adding positive twists}
\label{telecord}
\end{figure}

\section{Self-distributive cohomology of heaps}\label{sec:coh}

Computations of cocycle invariants rely on constructions of cocycles. 
Since heaps are our focus in this paper,  we start with developing computational methods of TSD cohomology groups for heaps, extensively utilizing group heap structures. 
These methods are of interest on its own for heap cohomology as well.
We present non-trivial  examples, and various subcomplexes that are useful in computations, as well as  having significance from point of view of framed link invariants.

\subsection{Non-triviality of cohomology}\label{sec:Nontrivial}

In this section we provide examples of non-trivial cohomology for a few abelian heaps.
The goal of this section is to motivate and prepare to set up more general computational methods in sections that follow. 
We also construct 2-cocycles for later use of cocycle invariants.
 
Let $X$ be a group heap.   Recall the 2-cocycle condition
 $$
(*) \quad 
\delta^2 \psi(x,y,z,u,v) = \psi(x,y,z)  - \psi(xu^{-1}v,yu^{-1}v,zu^{-1}v) - \psi(x,u,v)  + \psi(xy^{-1}z,u,v) = 0,
$$
where $x,y,z,u,v \in X$.

 \begin{lemma}\label{lem:deg2cocy}
 Let $X$ be a group heap and $A$ an abelian group.
If $\psi \in Z^2_{\rm SD}(X, A)$ is a 2-cocycle, then we have 
$\psi(x,y,y) = \psi(u,v,v)  $ for all $x, y, u, v \in X$.
Furthermore, $\psi=\sum_{ (x, y) \in X^2 } \chi_{(x,y,y)}$ is a 2-cocycle, i.e. $\psi \in  Z^2_{\rm SD}(X, A)$,
	where $\chi_{(x,y,z)} (x',y',z')=\delta(x,x') \delta(y,y') \delta(z,z')$ is the characteristic function with the Kronecker's $\delta$.
 \end{lemma}
 
\begin{proof}
Set $u=v$ in  $(*)$  to obtain
$ \psi(xy^{-1}z,u,u) = \psi(x,u,u)$, and by varying $y, z \in X$ and by variable change, we obtain  $ \psi (x, y, y)=\psi(xz, y, y) $, which implies $\psi(x,y,y) = \psi(u,y,y)$ for all $x,y,u\in X$.
Set $y=z$ in  $(*)$  to obtain 
$ \psi(x,y,y) = \psi(xu^{-1}v,yu^{-1}v,yu^{-1}v) $,  and by varying $u, v \in X$ and by variable change, we obtain $ \psi (x, y, y) = \psi (xz, yz, yz) $. 
Together with  $ \psi (x, y, y)=\psi(xz, y, y) $ we obtain
$ \psi (x, y, y) = \psi (x, yz, yz) $, which shows that $\psi(x,y,y) = \psi(x,v,v)$ for all $x, y, v\in X$. 
Now, combining the two results obtained shows that $\psi(x,y,y) = \psi(u,v,v)$ for all $x, y, u, v \in X$ as asserted, since $\psi(x,y,y) = \psi(u,y,y) = \psi(u,v,v)$. 

Let now $\psi=\sum_{ (x, y) \in X^2 } \chi_{(x,y,y)}$ be a $2$-cochain and let $(z, u_1, v_1, u_2, v_2)$ be a generic $3$-chain upon which we evaluate $\delta^2 \psi$. It is clear that if $u_i \neq v_i$ for $i=1,2$, then the $2$-cocycle condition is satisfied, since all the summands of $\delta^2\psi$ would vanish, since $\psi$ is a sum of characteristic functions of type $\chi_{(x,y,y)}$. When $u_i = v_i$ for at least one $i = 1,2$, we see directly that the summands of the $2$-cocycle condition vanish in pairs. 
Specifically, for $i=1$, $\psi(z,u_1,v_1) \neq 0$ if and only if $\psi(zu_2^{-1}v_2,u_1u_2^{-1}v_2,v_1u_2^{-1}v_2) \neq 0$, in which case their values are the same and 
appear with opposite signs. A similar argument applies  
for the two terms $\psi(z,u_2,v_2)$ and $\psi(zu_1^{-1}v_1,u_2,v_2)$. 
This completes the proof of the lemma. 
\end{proof}

\begin{lemma}\label{lem:core}
 Let $X$ be a group heap and $A$ an abelian group.
If $\psi \in Z^2_{\rm SD}(X, A)$ is a 2-cocycle, then we have 
$\psi(x, y,z)=\psi(x, z, z y^{-1} z ) $ for  all $x, y, z \in X$.
\end{lemma}

\begin{proof}
This is obtained by setting $y=u$ and $z=v$ in  $(*)$  and changing  variables.
\end{proof}

We note a curious fact that $(y, z ) \mapsto z y^{-1} z$ is a quandle operation called the {\it core} quandle (see, for example, \cite{FR}).

\begin{example}\label{ex:Z2}
	{\rm 
		Let $\Z_n$ be the integers modulo $n$ with abelian heap structure.
			We compute $H^2_{\rm SD}(\Z_2, \Z_n)$. 
By Lemma~\ref{lem:deg2cocy}, as the 2-cocycle condition for either the case $y=z$ or $u=v$ in  $(*)$  we have 
$$\psi(0, 0,0)=\psi(0, 1,1)=\psi(1,0,0)=\psi(1,1,1).$$
\begin{sloppypar}
\noindent
This implies that, when $\psi$ is expressed as a linear sum of characteristic functions $\chi_{(x,y,z)}$,
the coefficients of $\chi_{(0, 0,0)}$, $\chi_{(0, 1,1)}$, $\chi_{(1,0,0)}$, and $\chi_{(1,1,1)}$ are the same.
Next assume  that $y\neq z$ and $ u \neq v$. 
By Lemma~\ref{lem:core} we obtain $ \psi(x,0,1) = \psi(x,1,0)$ for $x =0,1$.
From the proof of Lemma~\ref{lem:core}, this is a consequence of the case $y=u$ and $z=v$.
If $y\neq z$, $u \neq v$  and $y \neq u$, then it follows that $z=u$ and $y=v$.
The same conditions follow from  $z \neq v$, $u \neq v$  and $y \neq u$.
Hence   the only remaining case is $z=u$ and $y=v$.
In  $(*)$  these conditions imply 
\end{sloppypar}
$$ 
\psi(x, y, z) - \psi (xz^{-1} y, y z^{-1} y, y)  - \psi(x, z, y)+ \psi (xy^{-1}z , z, y) =0 . 
$$
Together with $ \psi(x,0,1) = \psi(x,1,0)$, we obtain
$\psi (xy^{-1}z , z, y) = \psi (xz^{-1} y, y z^{-1} y, y)$. 
In $\Z_2$ this holds automatically. In summary the general solution to  $(*)$, 
written in terms of characteristic functions, $\psi$ is written as 
 		$$
		\psi = a (\chi_{(0,0,0)}+ \chi_{(0,1,1)} + \chi_{(1,0,0)}   +\chi_{(1,1,1)} ) + 
		b (\chi_{(0,0,1)} + \chi_{(0,1,0)}) + c(\chi_{(1,0,1)} + \chi_{(1,1,0)}). 
		$$
		Therefore $Z^2_{\rm SD}(\Z_2,\Z_n) = \Z_n^{\oplus 3}$. 
One computes 
$$( \delta^1 \chi_{(u)} )( x,y,z) =	\chi_{(u)} (d_2( x,y,z)) =	\chi_{(u)} ( (x) - (x-y+z) ) , $$
hence by checking all triples $(x,y,z)$ we obtain 
\begin{eqnarray*}
 \delta^1 \chi_{(0)} &=&  \chi_{( 0, 0, 1 )} + \chi_{( 0, 1, 0 )} - \chi_{(1, 0, 1  )} -  \chi_{( 1, 1, 0 )} , \\
  \delta^1 \chi_{(1)} &=& - \chi_{( 0, 0, 1 )} - \chi_{( 0, 1, 0 )} + \chi_{(1, 0, 1  )} +  \chi_{( 1, 1, 0 )} 
  \quad\! = \quad\! - \delta^1 \chi_{(0)} .
  \end{eqnarray*}
Thus we can rewrite $\psi$ as 
$$
		\psi = a (\chi_{(0,0,0)}+ \chi_{(0,1,1)} + \chi_{(1,0,0)}   +\chi_{(1,1,1)} ) + 
		(b+ c) (\chi_{(0,0,1)} + \chi_{(0,1,0)} ) - c	\  \delta^1 \chi_{(0)} 	.
$$
Let $c_0=(0,0,0)$, then $d_2(c_0)=(0)-(0)=0$ (where $(0)$ denotes a 1-chain)  so that $c_0 \in Z_2^{\rm SD}(\Z_2, \Z_n)$, and $\psi(c_0)=a$.
Let $c_1=(0,0,1)+ (1,0,1)$, then $d_2(c_1)=[(0)-(1)]+[(1)-(0)]=0$ so that $c_1 \in Z_2^{\rm SD}(\Z_2, \Z_n)$, 
and $\psi(c_1)=b+c$. Hence 
	we  obtained $H^2_{\rm SD}(\Z_2,\Z_n) \cong \Z_n \oplus \Z_n $. The two factors are generated by $\psi_0=\sum_{ (x, y) \in \Z_2 \times \Z_2} \chi_{(x,y,y)}$ and $\psi_1=\chi_{(0,0,1)}+\chi_{(0,1,0)}$.
		}
\end{example}

\begin{example}\label{ex:Z3}
	{\rm
		Let $\Z_3$ be endowed with abelian heap structure as before. We  
		compute $H^2_{\rm SD}(\Z_3,\Z_n)$. 
			By Lemma~\ref{lem:deg2cocy}, 
	we have that
$\psi_0 =\sum_{(x,y) \in \Z_3 \times \Z_3 }\chi_{(x, y, y)} $ is a 2-cocycle, and 
$\psi(0,0,0) =1$ for a 2-cycle $(0,0,0)$, so that this $\psi_0$ contributes $\Z_n$ to $H^2_{\rm SD}(\Z_3, \Z_n)$. 
Also from  Lemma~\ref{lem:deg2cocy}, for any given 2-cocycle $\psi'$, there is a constant $c$ such that $\psi=\psi'-c\psi_0$ satisfies 
$\psi(x,y,y)=0 $ for all $x, y \in \Z_3$, hence we assume this condition for $\psi$ below. 

By Lemma~\ref{lem:core}, we obtain $\psi (x, y, z)=\psi(x, z, 2z-y)$ for 
all $x, y, z\in \Z_3$. Recall from the Fox tricoloring that $\{y, z, 2z-y\}=\{0,1,2 \} $ for $y \neq z$  in $\Z_3$,
so that we obtain 
\begin{eqnarray}\label{eqn1}
\psi(x, 0, 1)=\psi(x, 1, 2)=\psi (x, 2,0)\quad &{\rm  and}& \quad  
 \psi(x, 0,2)=\psi(x, 2,1)=\psi (x,1,0).
 \end{eqnarray}
Hence $\psi$ is written as 
$$\psi=\sum_{x \in \Z_3}a_1(x)  [ \chi_{(x,0, 1)}+\chi_{(x,1, 2)}+\chi_{(x,2, 0)} ] 
 +  \sum_{x \in \Z_3}a_2(x) [ \chi_{(x,0,2) }+ \chi_{(x,2,1)}+\chi_{(x,1,0)}].$$
The equation   $(*)$   for the cases $z-y=v-u$
gives
$$
\psi(x, y,z)-\psi( x-y+z, y-y+z, z-y+z) -\psi( x, u, v)+ \psi(x-y+z, u,v) =0, 
$$
and  Equalities  (\ref{eqn1}) imply this equation under $z-y=v-u$,
since it holds that $\psi(x, y,z)=\psi( x, u, v)$ and $\psi(x-y+z, u,v)=\psi( x-y+z, z, 2 z-y)$.
 The cases 
$z-y \neq v-u$ remain. 
Since the cases $z-y=0$ or $v-u=0$ are checked under Lemma~\ref{lem:deg2cocy}, 
we have two cases: $z-y=1  \neq v-u =2$ and $z-y=2  \neq v-u =1$.
Under Equalities  (\ref{eqn1}) both cases reduce to 
$$ \psi(x, 0, 1)  + \psi (x+1, 0, 2 )   - \psi(x, 0, 2) - \psi(x+2, 0, 1)  = 0 . $$
Thus we obtain 
$ a_1 (x) + a_2(x+1) - a_2(x) - a_1(x+2)=0 . $
Adding the LHS of  these equations for all $x=0,1,2$ cancel all terms, hence 
these equations for $x=0$ and $x=1$ imply the equation for $x=2$. 
Hence the coefficients of $\psi$ are subject to $ a_1 (0) + a_2(1) - a_2(0) - a_1(2)=0 $
and $ a_1 (1) + a_2(2) - a_2(1) - a_1(0)=0 $.
Then $\psi$ is written as 
\begin{eqnarray*}
\psi&=&   a_1(0) [ \chi_{(0,0, 1)}+\chi_{(0,1, 2)}+\chi_{(0,2, 0)} ] 
 +a_1(1) [ \chi_{(1,0, 1)}+\chi_{(1,1, 2)}+\chi_{(1,2, 0)} ]  \\
& + & a_2(0) [ \chi_{(0,0,2) }+ \chi_{(0,2,1)}+\chi_{(0,1,0)}]
+ a_2(1) [ \chi_{(1,0,2) }+ \chi_{(1,2,1)}+\chi_{(1,1,0)}] \\
 &+&   (a_1(0)-a_2(0)+a_2(1) )  [ \chi_{(2,0, 1)}+\chi_{(2,1, 2)}+\chi_{(2,2, 0)} ] \\
 &+&
 (a_2 (1) + a_1(0) - a_1 (1) )  [ \chi_{(2,0,2) }+ \chi_{(2,2,1)}+\chi_{(2,1,0)}]
 \end{eqnarray*}
and we obtain $Z^2_{\rm SD}(\Z_3, \Z_n) \cong   \Z_n \oplus   \Z_n^{\oplus 4}$. 
One computes 
$$\delta^1\chi_{( u ) }=  \sum_{x=u, \,  u-y+z \neq  u }  \chi_{ ( x , y, z) }   -   \sum_{  x \neq  u, \, x -y+z = u ,  \, y\neq z } \chi_{(x , y, z) } . $$ 
We have $$\sum_{u \in \Z_3} \delta^1\chi_{( u ) }
= \sum_{ y \neq z }  \chi_{ ( x , y, z) }   -   \sum_{   y\neq z } \chi_{(x , y, z) } = 0 ,$$ so that 
$ {\rm Im} \delta^1$ has rank 2, generated by $\delta^1\chi_{( 0 ) }$ and $\delta^1\chi_{( 1 ) }$.
The  cocycle $\delta^1\chi_{( 0 ) }$ satisfies $a_1(0)=a_2(0)$ and all the other coefficients are identical,
and $\delta^1\chi_{( 1 ) }$ satisfies $a_1(1)=a_2(1)$ and all the other coefficients are identical.
Hence we  have $H^2_{\rm SD}(\Z_3, \Z_n)\cong  \Z_n \oplus   \Z_n^{\oplus 2}$. 
	}
\end{example}

We note that in both examples above there is a 
non-trivial cocycle
$\psi=\sum_{(x, y) \in X^2}\chi_{(x,y,y)}$ contributing one direct summand  to $H^2_{\rm SD}(X,A)$, and other cocycles 
involving  $\chi_{(x,y,z)}$ with $y\neq z$. This motivates us to define {\it degenerate} and {\it nondegenerate} subcomplexes in the next section for further computations.

\subsection{Degenerate and nondegenerate subcomplexes}\label{sec:degnondeg}

In this section we define degenerate subcomplexes and investigate the resulting long exact sequence,
which shows properties analogous to degenerate subcomplex in rack and quandle homology theories~\cite{LitherlandNelson}.

\begin{definition}\label{def:degenerate}
{\rm

Let $X$ be a heap.
Define the n-dimensional {\it degenerate} heap chain subgroup $C_n^{\rm DH}(X)$ to be the subgroup of
$C_n^{\rm SD} (X)$ generated by $(2n-1)$-tuples $(x_1, x_2, \ldots, x_{2n-1})$ such that 
$x_{2i}= x_{2i+1}$ 
for some $i >0$.

}
\end{definition}

\begin{lemma}\label{lem:degenerate}
The restriction of the boundary operator $d_n: C_n^{\rm SD} (X) \rightarrow  C_{n-1}^{\rm SD} (X)$ 
on  $C_n^{\rm DH}(X)$ defines a subcomplex $(C_n^{\rm DH}(X) ,d_n)$, which we call the {\it degenerate subcomplex}. 
\end{lemma}

\begin{proof}
Let $(x_1, \ldots,  x_{2n-1})$ be  an $n$-chain  such that $x_{2j} = x_{2j+1}$ for some $j = 1, \ldots , n-1$. By definition of ternary self-distributive differential $d_n$ we have
\begin{eqnarray*}
\lefteqn{ d_n (x_1, x_2 , \ldots , x_{2n-1}) =}\\
& &  \sum_{i=1}^{n-1} (-1)^i [\ (x_1, \ldots, \widehat{x_{2i},x_{2i+1}},\ldots , x_{2n-1}) - (x_1 x_{2i}^{-1}x_{2i+1},  \ldots, \widehat{x_{2i},x_{2i+1}},\ldots  ,x_{2n-1}) \ ] .
\end{eqnarray*}
For all $i$'s such that  $x_{2i} \neq x_{2i+1}$ we have that both terms 
$$(x_1, \ldots, \widehat{x_{2i},x_{2i+1}},\ldots , x_{2n-11})\quad {\rm  and } \quad (x_1 x_{2i}^{-1}x_{2i+1},  \ldots, \widehat{x_{2i},x_{2i+1 }},\ldots  , x_{2n-1}) $$
 are in $C_{n-1}^{\rm DH}(X)$, since there exists some $j\neq i$ such that $x_{2j} = x_{2j+1}$. For all $i$'s such that $x_{2i-1} = x_{2i}$ we have that 
 $$(x_1, \ldots, \widehat{x_{2i},x_{2i+1}},\ldots , x_{2n-1} ) \quad {\rm  and } \quad (x_1 x_{2i-1 }^{-1}x_{2i},  \ldots, \widehat{x_{2i},x_{2i+1}},\ldots  ,x_{2n-1}) $$ coincide. Thus the differential $d_n$ restricts to a well defined differential on $C_n^{\rm DH}(X)$. 
\end{proof}

\begin{definition}\label{def:nondegenerate}
	{\rm 
  Let us consider 
  a similar 
   situation as in Definition~\ref{def:degenerate} and define the 
  {\it nondegenerate heap chain subgroup}  $C_n^{\rm NDH}(X)$ to be the subgroup of $C_n^{\rm SD}(X)$ generated by $n$-chains $(x_1, x_2, \ldots, x_{2n-1})$ where $x_{2i} \neq x_{2i+1}$ for all $i = 1, \ldots , n-1$. 
}
\end{definition}

\begin{lemma}\label{lem:nondegenerate}
 The boundary operator $d_n$ restricts to a well defined differential on $C_n^{\rm NDH}(X)$ therefore defining a subcomplex, which we call the  {\it nondegenerate subcomplex}. 
\end{lemma}
\begin{proof}
	Observe that $(x_1x_{2i}^{-1}x_{2i+1},\ldots , x_{2i-2}x_{2i}^{-1}x_{2i+1}, x_{2i-1}x_{2i}^{-1}x_{2i+1}, \widehat{x_{2i},x_{2i+1}},\ldots , x_{2n-1} )$ is an element of $C_{n-1}^{\rm NDH}(X)$ for all $i$'s since $x_{2j}x_{2i}^{-1}x_{2i+1} = x_{2j+1}x_{2i}^{-1}x_{2i+1}$ if and only if $x_{2j} = x_{2j+1}$. 
\end{proof}

Then the long exact sequence of 
corresponding to the short exact sequence of complexes
 $$
0 \rightarrow C_\bullet^{\rm DH}(X) \rightarrow  C_\bullet^{\rm SD} (X) \rightarrow  C_\bullet^{\rm SD} (X) / C_\bullet^{\rm DH}(X) \rightarrow 0
$$
decomposes into split short exact sequences 
$$
0 \rightarrow C_n^{\rm DH}(X) \rightarrow  C_n^{\rm SD} (X) \rightarrow  C_n^{\rm SD} (X) / C_n^{\rm DH}(X) \rightarrow 0.
$$
 Since $C_n^{\rm H} (X) / C_n^{\rm DH}(X)$ is isomorphic to the nondegenerate subcomplex $C_n^{\rm NDH}(X)$ and, moreover, the differentials respect the splitting at each $n$ from Lemmas~\ref{lem:degenerate} and~\ref{lem:nondegenerate},
we obtain a split short exact sequence 
$$0 \rightarrow H_n^{\rm DH}(X) \rightarrow  H_n^{\rm SD} (X) \rightarrow  H_n^{\rm NSD} (X) \rightarrow 0.
$$
This implies that the self-distributive homology groups decompose as the direct sum of degenerate and nondegenerate homologies
as follows. 

\begin{proposition}\label{prop:DH}
Let $X$ be a heap. Then ternary self-distributive homology is given by 
$$
H_n^{\rm SD}(X) \cong H_n^{\rm DH}(X)\oplus H_n^{\rm NDH}(X).
$$
\end{proposition}

\begin{definition}
	{\rm 
By dualization, given an abelian group $A$, we obtain degenerate and nondegenerate cohomologies, indicated by $H^n_{\rm DH}(X,A)$ and $H^n_{\rm NDH}(X,A)$, respectively. As in the case of homology, ternary self-distributive cohomology splits into direct sum of degenerate and nondegenerate parts. 
}
\end{definition}

We determine the degenerate cohomology of dimension 2.

\begin{proposition}\label{prop:degcohomology}
	Let $X$ be a finite group heap  and $A=\Z_n$. 
	Then the degenerate second cohomology group with coefficients in $A$ is given by
	$$
	H^2_{\rm DH}(X,A) \cong A . 
	$$ 
\end{proposition}

\begin{proof}
From Lemma~\ref{lem:deg2cocy}, 
it follows  that $\psi=\sum_{(x,y)\in X^2 } \chi_{(x,y,y)}$ is a 2-cocycle in $Z^2_{\rm DH}(X,A)$.
For $u \in X$, $(u,u,u)$ is a 2-cycle and  $\psi (u,u,u)=1$. Hence   $H^2_{\rm DH}(X,A) \cong A $ generated by $[\psi]$. 
\end{proof}

Next we show non-triviality of nondegenerate second cohomology for group heaps  with elements of even order.

 \begin{lemma}\label{lem:2cocy}
 Let $X$ be a ring regarded as an abelian group heap by ring addition,  and let $a, b, c \in X$. 
 Let $\psi_{(a,b,c)} (x, y, z) =(ax + b(z-y) + c ) (z-y) $. Then $\psi_{(a,b,c)} $  is an $X$-valued 2-cocycle, $\psi_{(a,b,c)}  \in Z^2_{\rm SD} (X, X)$.
 \end{lemma}

\begin{proof}
One computes the positive terms of the 2-cocycle condition as 
 \begin{eqnarray*}
\lefteqn{ \psi_{(a,b,c)}   (x, y,z ) + \psi (x - y + z   ,  u, v) }\\
&=&
 (ax + b (z-y) +c  )(z-y) + [ a (x-y+z) + b (v-u) + c  ] (v-u)   \\
 &=&  a [x(z-y) + x(v-u) + (z-y)(v-u)] + b [(z-y)^2 + (v-u)^2] + c [(z-y)+(v-u) ]  .
\end{eqnarray*}
The other terms are
 \begin{eqnarray*}
\lefteqn{\psi_{(a,b,c)}   (x, u,v ) + \psi (x -u+v , y-u+v, z-u+v) }\\
&=&
( a x + b (v-u) + c )  (v-u)  + [ a ( x-u+v) + b (z-y)+ c  ]  (z-y) \\
&=&  a [x(z-y) + x(v-u) + (z-y)(v-u)]  + b [(z-y)^2 + (v-u)^2] + c [(z-y)+(v-u)]
\end{eqnarray*}
as desired.
\end{proof}

\begin{lemma}\label{lem:2cycle}
Let $X$ be an abelian heap.
 Let $x \in X$,  and let $w \in X$ be of order $k$, i.e. $k w =0$.
For $(y, z ) \in X \times X$ with $w=z-y$, the 2-chain
$ c(x, y,z)= \sum_{i=0}^{k-1} (x + i w ,  y,z )$ is a 2-cycle,
$ c(x, y,z) \in Z_2^{\rm SD}(X)$. 
\end{lemma}

\begin{proof} Recall that $(x)$ denotes a 1-chain, and that 
 $d_2 ( ( x, y,z) ) = (x ) - (x -y+z)=(x)- (x+w)$.
Then one computes
$$d_2 (c(x, y,z) ) =
[ (x) - (x+w) ] + [ (x+w) - (x+2w)] + \cdots + [ (x+ (k-1)w) - (x+ kw) ]=  0,
$$
as desired. 
\end{proof}

\begin{proposition}\label{prop:nondeg}
Let $X$ be any abelian heap with an element of an even order. Then  we have $H^2_{\rm NDH}(X,X)\neq 0$. 
\end{proposition}

\begin{proof}
Let $w \in X$ be an element of even order, 
and 
$k$ the minimum positive even  integer 
that satisfy $kw=0$.
We evaluate  the 2-cycle in Lemma~\ref{lem:2cocy} by the 2-cocycle in Lemma~\ref{lem:2cycle}:
\begin{eqnarray*}
 \psi_{(a,b,c)} (c (x,y,z) ) &=&  \sum_{i=0}^{k-1} \psi_{(a, b, c)} ( x+i w,  y,z )=
 \sum_{i=0}^{k-1} ( a (x+iw) + b y +c)w   \\
&  = & ( a k x + a ( k(k-1) /2) w + b k w + ck  ) w = a [ k x +  ( k (k-1) / 2 ) w ] .
\end{eqnarray*}
If $k$ is even, then $k-1$ is odd, and $(k(k-1)/2)w = (k/2) (k-1) w  \neq 0 \in X$, so that a choice of $a\neq 0$ and $x=0$  in Lemma~\ref{lem:2cocy} gives a non-trivial evaluation. Hence  $\psi_{(a,b,c)} $  ($a \neq 0$) is non-trivial in $H^2_{\rm NDH}(X,X)$.
\end{proof}


\subsection{Construction of $2$-cocycles}\label{sec:cocy}

In this section we give $2$-cocycles for cyclic and dihedral group heaps, motivated by 
computations in Example~\ref{ex:Z3} and Appendix~\ref{app:exactseq}. 
These cocycles are proved non-trivial and mostly   linearly independent  in Proposition~\ref{prop:cyclic/dihedral}, and used  to obtain lower bounds of ranks of $2$-dimensional cohomology groups for these heaps.

\begin{lemma}\label{lem:Zn}
Let $\Z_n=\langle \ r \mid r^n=1 \  \rangle$ be the cyclic group of order $n$ 
in multiplicative notation with a generator $r$.
Let $\phi_i=\sum_{x \in \Z_n } [ \sum_{j=0}^{n-1}  \chi_{(x,r^j,r^{j+i})} ] $, $i=1, \ldots, n-1$.
Then $\phi_i$ is a nondegenerate $2$-cocycle, $\phi_i \in C^2_{\rm NDH}(\Z_n, \Z)$, for all 
$i=1, \ldots, n-1$.
\end{lemma}

\begin{proof}
For a fixed $i$, the 2-cocycle $\phi_i$ vanishes for 2-chains $(x, u, v ) \in C_2^{\rm NDH} (X, \Z)$ if 
$v \neq u r^i$. 
Hence if $v\neq u r^i$, then the last two terms of $(*)$, $- \phi(x, u, v) +  \phi(x y^{-1} z , u, v) $,  both vanish.
If $v=ur^i$, then  both terms are $1$ and cancel. Hence we focus on the first two terms.

Let $v=ur^k$, then the first two terms of $(*)$ are 
$\phi(x, r^j, r^{j+m}) - \phi(x r^{k}, r^{j+k}, r^{j+m+k})$ for some $j, m \in \Z_n$. 
If $m=i$, then both terms are 1 and cancel. 
If $m\neq i$, then  both vanish. 
%
Hence $(*)$ holds.
\end{proof}

\begin{lemma} \label{lem:Dn}
Let $D_n$ be the dihedral  group of order $2n$ generated by a rotation $r$ and reflection $a$ with a relation $ ara = r^{-1}$ as before. 
Let $\psi_i=\sum_{x \in D_n } [ \sum_{j=0}^{n-1} (  \chi_{(x,r^j,r^{j+i})} +\chi_{ (x, ar^{-j},ar^{-j-i}  ) } )  ] $, $i=1, \ldots, n-1$.
Then $\psi_i$ is a nondegenerate $2$-cocycle, $\psi_i \in C^2_{\rm NDH}(D_n, \Z)$, for all 
$i=1, \ldots, n-1$.
\end{lemma}

\begin{proof}
We use the description  of the  dihedral group    $D_n = F \rtimes G$,
	where   $F:=\braket{r} \cong  \Z_n$ is  the subgroup of rotations generated by $r$ and $G := \braket{a}\cong \Z_2$. We proceed as in the proof of  Lemma~\ref{lem:Zn}.
Note that $\psi_i(x,u,v)=0$ unless 
$(x, u, v)=(x, r^j,r^{j+i})$ or $(x, ar^{-j},ar^{-j-i}  )$ for some $j$.
If $(x, u, v)=(x, r^j,r^{j+i})$ or $(x, ar^{-j},ar^{-j-i}  )$, then the last two terms of $(*)$ 
are $1$ and cancel. Otherwise both vanish, so that we focus on the first two terms.

Since  elements of $D_n$ are written as $r^k$ or $a r^k$ for some $k$, either $v=u  r^k$
or  $v=u a r^k$ holds. 
If  $v=u  r^k$, then the first two terms of $(*)$ are 
$\phi(x, r^j, r^{j+m}) - \phi(x   r^{k},   r^{j+k},   r^{j+m+k})$ for some $j, m \in \Z_n$.
If $m=i$, then both terms are 1 and cancel. 
If $m \neq i$, then both vanish.

If $v=u a r^k$, then 
the first two terms are 
$\phi(x, r^j, r^{j+m}) -  \phi(x  a r^{k},  r^{j} a r^k ,   r^{j+m}  a r^k )$, where 
the second term 
is $  \phi(x  a r^{k},  a r^{k-j} ,   a r^{k-j-m} ) $.
If $m=i$, then both are 1 and cancel, and if $m\neq i$, then both vanish. 
Hence $(*)$ holds. 
\end{proof}

\subsection{Coset subcomplexes}\label{sec:coset}

In this section we introduce subcomplexes defined by means of cosets, and use them to compute 2-dimensional cohomology groups for dihedral group heaps. 
We note that left cosets of a group are subheaps.

The coset subcomplexes are characteristic to group heaps, making use of both TSD and group 
structures, and useful in estimating ranks of cohomology groups as presented here, and also potentially 
useful for producing 2-cocycles that can be used for invariants discussed in Section~\ref{sec:cocyinv}.
Although the constructions of cocycles in Section~\ref{sec:cocy} were motivated by considerations in this section, the computations  and results in this section are not directly used 
in the rest of the paper, so that the reader interested in framed link invariants could proceed to the next section.

\begin{definition}
{\rm
Let $X$ be a group heap, and let $G$ be a subgroup (with respect to the group operation). 
Define the $n$-dimensional {\it chain group localized at $G$}, denoted by  $C_n^{ \{ G \} } (X)$, to be the subgroup of 
$C_n^{ \rm SD} (X)$ generated by $n$-chains $(x_1, \ldots, x_{2n-1})$ such that 
$x_{2i}$ and $x_{2i+1}$ belong to the same left coset of $G$, that is, 
$x_{2i}G = x_{2i+1}G$, for all $i=1, \ldots, n-1$. 
}
\end{definition}

\begin{lemma}
Let $X$ be a group heap and $G$ its subgroup.
Then the self-distributive boundary operator $d_n: C_n^{\rm SD} \rightarrow C_{n-1}^{\rm SD}(X)$ restricted to $C_n^{ \{ G \} } (X)$ defines a subcomplex 
$(C_n^{  \{ G \} } (X), d_n)$.
\end{lemma}

\begin{proof}
Let $(x_1,\ldots , x_{2n-1})$ be an $n$-chain in $C_n^{ \{ G \} } (X)$. It is clear that for all $i = 1,\ldots ,n-1$ the term $(x_1,\ldots , \widehat{x_{2i},x_{2i+1}}, \ldots , x_{2n-1})$ is an $(n-1)$-chain in $C_{n-1}^{ \{ G \} } (X)$. 
Observe that if $u$ and $v$ are in the same left $G$-coset, then $w \in X$ and  $wu^{-1}v$ are in the same left $G$-coset,  since $u^{-1}v \in G$.  
It follows that $(x_1x_{2i}^{-1}x_{2i+1},  \ldots , x_{2i-2}x_{2i}^{-1}x_{2i+1}, x_{2i-1}x_{2i}^{-1}x_{2i+1},\widehat{x_{2i},x_{2i+1}}, \ldots , x_{2n-1})$ is an $(n-1)$-chain in $C_{n-1}^{ \{ G \} } (X)$ since $x_{2j}x_{2i}^{-1}x_{2i+1}$ and $x_{2j}$ are in the same left $G$-coset for each $j = 1,\ldots  , i-1$, and similarly for $x_{2j+1}x_{2i}^{-1}x_{2i+1}$ and $x_{2j+1}$. It follows that $x_{2j}x_{2i}^{-1}x_{2i+1}$ and $x_{2j+1}x_{2i}^{-1}x_{2i+1}$ are in the same left $G$-coset. This implies that the boundary operator restricted to $C_n^{ \{ G \} } (X)$ gives a subcomplex.
\end{proof}

\begin{definition}
{\rm
Let $X$ be a group heap and $G$ its subgroup. 
We denote the intersection $C^{\rm NDH}_n (X)  \cap C^{ \{ G \} }_n (X) $ by ${C}^{ {\rm N} \{ G \} }_n (X)$.
}
\end{definition}

\begin{lemma}\label{lem:relativehom}
The restriction of the boundary operator $d_n: C_n^{\rm SD} \rightarrow C_{n-1}^{\rm SD}(X)$ on
${C}^{  {\rm N}  \{ G \} }_n (X)$ induces the subcomplex $( {C}^{ {\rm N}  \{ G \} }_n (X), d_n)$. Moreover, if $G$ is a 
nontrivial subgroup of $X$, $( {C}^{ {\rm N}  \{ G \} }_n (X), d_n)$ consists of nontrivial groups ${C}^{ {\rm N}  \{ G \} }_n (X)$. 
\end{lemma}

\begin{proof}
$( {C}^{  {\rm N}  \{ G \} }_n (X), d_n)$ is a subcomplex because so are both $C^{\rm NDH}_n (X)$ and $C^{ \{ G \} }_n (X) $. If $G$ is a nontrivial subgroup, there exists at least one pair of elements $y \neq z$ in the same left $G$-coset. This implies that the chains $(e, y,z, \cdots, y,z)$ are nontrivial, where the pair $y, z$ is repeated $n-1$ times, therefore showing that $( {C}^{ {\rm N}  \{ G \} }_n (X), d_n)$ is nontrivial for all $n\in \mathbb N$. 
\end{proof}

\begin{definition}\label{def:relativehom}
	{\rm 
The homology corresponding to $(  {C}^{ {\rm N}  \{ G \} }_n (X), d_n)$, according to Lemma~\ref{lem:relativehom}, is denoted by $H_n^{ {\rm N}  \{G\}}(X)$ and it is called the {\it nondegenerate homology of $X$ relative to $G$}. By dualization, for a given abelian group $A$, we have cohomology groups $H^n_{ {\rm N}  \{G\}}(X,A)$ which we call {\it nondegenerate cohomology of $X$  localized at  
$G$}. 
}
\end{definition}

\begin{definition}\label{def:relativecohom}
	{\rm
	A standard (co)homological argument gives relative (co)homology groups corresponding to nondegenerate (co)homology localized at 
	$G$, by considering the quotient complex
	$( C_n^{\rm NDH}(X) / C_n^{ {\rm N} \{ G \} } (X), d_n)$.
	 These (co)chain complexes are denoted $\widehat C_n^{{\rm N} \{G\}}(X)$ and $\widehat C^n_{{\rm N} \{G\}}(X,A)$, respectively, and their (co)homology groups are denoted by the symbols $\widehat H_n^{{\rm N} \{G\}}(X)$ and $\widehat H^n_{{\rm N} \{G\}}(X,A)$. }
\end{definition}

A direct characterization of nondegenerate 
	 relative $n$-cocycles localized at $G$ 
	   is as follows. An $n$-cochain $\psi$ is in $\widehat{C}^n_{{\rm N} \{G\}}(X,A)$ if and only if it is zero when evaluated on chains in $ C_n^{{\rm N} \{G\}}(X)$, and satisfies the $n$-cocycle condition.

\begin{proposition}\label{prop:longexact}
	Localized cohomology and relative cohomology are related to nondegenerate cohomology by the long exact sequence
	$$
	\cdots \rightarrow \widehat{H}^n_{{\rm N} \{G\}}(X,A) \xrightarrow{j^*} H^n_{\rm NDH}(X,A) \xrightarrow{i^*} H^n_{{\rm N} \{G \}}(X,A) \xrightarrow{\delta} \widehat{H}^{n+1}_{{\rm N} \{ G \}}(X,A) \rightarrow \cdots 
	$$
	Moreover, the first connecting morpshim $\delta$ is zero and $\widehat{H}^2_{{\rm N} \{G\}}(X,A)$ injects into $H^2_{\rm NDH}(X,A)$.
\end{proposition}
\begin{sloppypar}
\begin{proof}
	To see that the first connecting morphism is trivial, 
we	observe that $ C^1_{{\rm N} \{G\}} (X,A)= C^1_{\rm NDH}(X,A)$. 
\end{proof}
\end{sloppypar}

\begin{proposition}
\label{prop:equivariance}
	Let $G\leq X$ be a subgroup of $X$. Then relative second 
		cocycles  $\phi \in  \widehat{Z}^2_{{\rm N} \{G\}}(X,A)$  are equivariant with respect to action of $G$ defined by componentwise multiplication, 	$$
	\phi(x,y,z) = \phi(xg,yg,zg),
	$$
	for all $x,y,z\in X$ and all $g\in G$ with $y\neq z$ in different $G$-cosets.
\end{proposition}

\begin{proof}
	This follows from $2$-cocycle condition $(*)$ with $u= 1$, $v\in G$ and $z$, $y$ in different $G$-cosets, since we obtain from $(*)$
	$$
	\phi(x,y,z) - \phi(xv, yv, zv) - \phi(x,1,v) + \phi(xy^{-1}z, 1, v) = 0,
	$$
	which reduces to the equivariance since $\phi$ is zero when evaluated on chains with $u,v$ in the same $G$-coset. 
\end{proof}

Proposition~\ref{prop:longexact} presents a useful way of breaking nondegenerate cohomology into smaller parts easier to compute.

\begin{example}\label{ex:Z4}
{\rm
Let $X=\Z_4$, $\Z_2 \cong G=\{ 0, 2 \} < X$ and $A=\Z$. 
A cocycle $\phi \in    \widehat{Z}^2_{{\rm N} \{G\}}(X,A) $ satisfies the $2$-cocycle condition $(*)$ and
$\psi(x, y, y)=0$ for all $x, y \in X$ (nondegenerate condition) and 
$\phi(x, 0, 2)=\phi(x,2,0)=0$, $\phi(x, 1,3)=\phi(x,3,1)=0$ for all $x\in X$ (the localized quotient condition, the definition of   $\widehat{Z}$). 
Proposition~\ref{prop:equivariance} implies
$\phi(x,y,z)=\phi(x+2,y+2,z+2)$ for all $x,y,z \in X$. 
Further computations in Section~\ref{app:exactseq} implies  $\widehat{H}^2_{ {\rm N} \{ G \} } (X, A)\cong A $.
Another  computation in Appendix~\ref{app:exactseq} shows that 
$H^2_{ {\rm N} \{ G \} } (X, A) \cong A^{\oplus 6}$.
Hence from Proposition~\ref{prop:longexact} we obtain the following exact sequence:
$$ 0 \rightarrow \widehat{H}^2_{{\rm N} \{G\}}(X,A) 
\stackrel{j^*}{\rightarrow} 
H^2_{ {\rm NDH} } (X, A)
\stackrel{i^*}{\rightarrow} 
i^*( H^2_{ {\rm N} \{ G \} } (X, A) ) 
 \rightarrow  0
$$
where $i^*( H^2_{ {\rm N} \{ G \} } (X, A) ) $ is isomorphic to $A^{\oplus r}$ with $r \leq 6$. 
Hence 
$H^2_{ {\rm NDH} } (X, A)$ is free of rank $\leq 7$. 
Detailed computations are included in Appendix~\ref{app:exactseq}.

}
\end{example}

 This procedure
 of localization
  can be iterated in order to compute $\widehat{H}^n_{{\rm N} \{G\}}(X,A)$ 
  as follows. 

Suppose $G,F \leq X$ are subgroups of $X$, and $C_n^{\rm NDH}(X)$ has been localized at $G$ to obtain complexes $C_n^{\rm N\{G\}}(X)$ and $\widehat C_n^{\rm N\{G\}}(X)$. We define now subgroups of relative chains $\widehat C_n^{\rm N\{G\}}(X)$ by the quotient 
$$
\frac{C_n^{\rm N\{F\}}(X)}{C_n^{\rm N\{F\}}(X)\cap C_n^{\rm N\{G\}}(X)} \cong \frac{C_n^{\rm N\{F\}}(X)+C_n^{\rm N\{G\}}(X)}{C_n^{\rm N\{G\}}(X)}\leq \widehat C_n^{\rm N\{G\}}(X).$$

The groups so defined constitute a subcomplex of $\widehat C_n^{\rm N\{G\}}(X)$, with the differential obtained from that of $\widehat C_n^{\rm N\{G\}}(X)$ by restriction. This subcomplex is denoted by the symbol $C_n^{\rm N\{G,F\}}(X)$, and its homology by $H_n^{\rm N\{G,F\}}(X)$. We therefore obtain a long exact sequence computing relative homology $\widehat H_n^{\rm N\{G\}}(X)$, where the relative homology of $H_n^{\rm N\{G,F\}}(X)$ is defined explicitly by the chain complex 
$$
\widehat C_n^{\rm N\{G,F\}}(X) := \frac{\frac{C_n^{\rm NDH}(X)}{C_n^{\rm N\{G\}}(X)}}{\frac{C_n^{\rm N\{F\}}(X)+C_n^{\rm N\{G\}}(X)}{C_n^{\rm N\{G\}}(X)}} \cong \frac{C_n^{\rm NDH}(X)}{C_n^{\rm N\{F\}}(X)+C_n^{\rm N\{G\}}(X)}.
$$

\begin{remark}\label{rmk:semidirect} 
	{\rm 
		Observe that if subgroups $G$ and $F$ of $X$  satisfy that if $y\neq z$ are in the same $F$-coset, then $y,z$ are in different $G$-cosets, 
				the description of groups 	$C_n^{\rm N\{G,F\}}(X)$ becomes simpler as $C_n^{\rm N\{G,F\}}(X)$ coincide with 
		$C_n^{\rm N\{F\}}(X)$, 
		since $C_n^{\rm N\{F\}}(X)\cap C_n^{\rm N\{G\}}(X) = 0$ for all $n$. 
		This is the case, for example, when X is a semidirect product $X = F\rtimes G$. In fact let $(f_1,g_1) \cdot (f_2,g_2) = (f_1\phi_{g_1}(f_2),g_1g_2)$ denote the product operation in $X$, where $\phi$ indicates the automorphism $G \longrightarrow {\rm Aut}(F)$ determining the semidirect product structure. Let $(x,y) \neq (x',y')$ be in the same $G$-coset, where $G$ is identified with the subgroup $1\times G$ of $X$. Then we have the equality of sets $(x\phi_y(1)\times y)G = (x'\phi_{y'}(1)\times y')G$, from which $x = x'$. If $(x,y)$ and $(x',y')$ are in the same $F$-coset as well, then it is easily seen that it follows that $y = y'$, against the fact that $(x,y) \neq (x',y')$ by assumption. Therefore we have 
		$C_n^{\rm N\{F\}}(X)\cap C_n^{\rm N\{G\}}(X) = 0$ for all $n$
				 as claimed.  
		
	}
\end{remark}
	
	By dualization, we obtain an iterated localization for cochain complexes, and associated cohomologies that inherits sub/superscirpts as in Definition~\ref{def:relativehom}, in cohomological notation.  The cocycles are explicitly described as follows. A cochain $\phi\in  C^n_{\rm N\{G,F\}}(X)$ is an $n$-cocycle if and only if it vanishes on chains localized at $G$, and satisfies the $n$-cocycle condition on chains localized at $F$. A cochain $\psi\in \widehat C^n_{\rm N\{G,F\}}(X)$ is a cocycle by definition, if and only if it vanishes on chains localized either at $F$ or at $G$, and satisfies the $n$-cocycle condition.
	
	\begin{remark}\label{rmk:injection}
		{\rm 
	Applying Proposition~\ref{prop:longexact} we see that relative cohomology of iterated localization $\widehat H^2_{\rm N\{G,F\}}(X,A)$ injects into $H^2_{\rm NDH}(X,A)$, since $\widehat H^2_{\rm N\{G,F\}}(X,A) \hookrightarrow \widehat H^2_{\rm N\{G\}}(X,A) \hookrightarrow H^2_{\rm NDH}(X,A)$. A special case is when $X = F\rtimes G$, and therefore using Remark~\ref{rmk:semidirect} we have $C^n_{\rm N\{G,F\}}(X,A) \cong C^n_{ {\rm N} \{F\}}(X,A)$, so that $H^2_{\rm N\{G,F\}}(X,A) = H^2_{\rm N\{F\}}(X,A)$, from which we get an injection of $\widehat H^2_{\rm N\{F\}}(X,A)$ into $H^2_{\rm NDH}(X,A)$. Moreover, since relative second cohomology is equivariant, this observation provides a useful way of constructing nontrivial cohomology classes in $H^2_{\rm NDH}(X,A)$ from computations on smaller chain complexes. 
	}
	\end{remark}
	
	We utilize iterated localizations in the following example. 
	
\begin{example}\label{ex:D3}
	{\rm 
		Let $X=D_3$ be the dihedral group  of order $6$ and let $A = \Z$. 
By Proposition~\ref{prop:degcohomology} it follows that $H^2_{\rm DH}(X,A) \cong A$.  
We 
consider the nondegenerate cohomology of $X$. 

Recall  that the dihedral group $X=D_3$ has a presentation 
$\braket{a,r\ | \ a^2 = r^3 =1,\ ara = r^{-1}}$, where $a$ represents a reflection and $r$ represents a rotation. Consider the subgroup $G = \{1,a\}$ generated by reflection $a$. Considering the coset complex localized at $G$ we obtain from Proposition~\ref{prop:longexact} a long exact sequence relating nondegenerate cohomology of $X$ to localization at $G$. In particular in dimension 2 we have
$$
0 \rightarrow \widehat{H}^2_{{\rm N} \{G\}}(X,A) \xrightarrow{j^*} H^2_{\rm NDH}(X,A) \xrightarrow{i^*} H^2_{{\rm N} \{G \}}(X,A) \xrightarrow{\delta} \widehat{H}^3_{{\rm N} \{ G \}}(X,A) \rightarrow \cdots .
$$
Let $F = \{1, r, r^2\}$ indicate the subgroup generated by rotations. Since if $x, y$ are in the same left $F$-coset it follows that $x$ and $y$ are not in the same left $G$-coset. We iterate the procedure of localization at $F$, to compute the relative cohomology to $G$, as observed in the paragraph above this example. So we obtain a second long exact sequence that at dimension 2 takes the form
$$
0 \rightarrow \widehat{H}^2_{{\rm N} \{G,F\}}(X,A) \xrightarrow{j^*} \widehat H^2_{\rm N\{G\}}(X,A) \xrightarrow{i^*} H^2_{{\rm N} \{G,F \}}(X,A) \xrightarrow{\delta} \widehat{H}^3_{{\rm N} \{ G,F \}}(X,A) \rightarrow \cdots
$$
using the same notation as above. 
As in Remark~\ref{rmk:injection} we have that $\widehat{H}^2_{{\rm N} \{G,F\}}(X,A)$ injects into $H^2_{\rm NDH}(X,A)$ since the morphisms $j^*$ in the two long exact sequences are both injections.

\begin{sloppypar}
Computation in Appendix~\ref{app:exactseq} shows that  $\widehat{H}^2_{{\rm N} \{G,F\}}(X,A) = 0$, from which it follows that $\widehat{H}^2_{{\rm N} \{G\}}(X,A) \cong i^* (\widehat{H}^2_{{\rm N} \{G\}}(X,A)) \leq H^2_{{\rm N} \{G,F \}}(X,A)$.
It is also computed that $\widehat H^2_{\rm N\{G\}}(X,A)$ has rank $2$ and that
 ${\rm rank}\ H^2_{\rm N\{G\}}(X,A) \leq 9$. Then  $H^2_{\rm NDH}(X,A)$ corresponds to a 
 extension 
 \end{sloppypar}
$$
0 \longrightarrow \widehat{H}^2_{{\rm N} \{G\}}(X,A) \longrightarrow H^2_{\rm NDH}(X,A)  \longrightarrow i^*(H^2_{\rm NDH}(X,A)) \longrightarrow 0,
$$
and it has rank at most $11$.
See Appendix~\ref{app:exactseq} for details.

}
\end{example}

\section{The fundamental heap of framed links}\label{sec:fund}

In this section we define and study the fundamental heap of framed links. The definition is analogous to presentations of knot groups and quandles, and defined by generators assigned to double arcs and relations assigned to crossings.
Then colorings defined in Section~\ref{sec:color} can be regarded as heap homomorphism from the fundamental heap to a given group heap. 
We relate the fundamental heap to Vinberg groups, and Wirtinger presentations of the knot group.

\subsection{Definitions and examples}\label{sec:hdef}

First we present definitions and examples.

\begin{definition}\label{def:fund}
{\rm
The {\it fundamental heap} $h(L)$ of an unoriented  framed link $L$ is defined as follows.
Let $D$ be a diagram of  $L$ with single arcs with blackboard framing. 
We define $h(L)$ by a presentation using $D$ and show that it is well-defined.
Each single arc  in Figure~\ref{crossing} (A) represents  double (parallel) arcs as in Figure~\ref{crossing} (B).
	Let ${\cal  A}$ be the set of doubled arcs.
Each of double arcs is assigned a generator. In the figure, generators are represented by letters 
(labels) $x,y,u,v,z,w$. Letters (labels) assigned to arcs are identified as (the names of) arcs themselves, and 
regarded as elements of ${\cal A}$. 
Then the set of generators of $h(L)$ is ${\cal A}$, which is identified with letters assigned.

For each crossing, a pair of relations is defined.
In Figure~\ref{crossing}, the relations are defined as $\{ z=x u^{-1} v, w=y u^{-1} v \}$. 
Specifically, when the arc $x$ goes under the arcs $(u, v)$, in this order, to the arc $z$, then 
the relation is defined as $ z=x u^{-1} v$, and similar from $y$ to $w$.
The set of union of the two relations over all crossings is denoted by ${\cal T}$ and constitutes 
the relation of $h(L)$.

The fundamental heap $h(D)$ is the group heap defined by the group presentation
with a set of generators corresponding to double arcs, and the set of relations assigned to all crossings:
$
\langle
\ 
{\cal A}'  \ | \ 
{\cal T} 
\
\rangle 
$.
}
\end{definition}

\begin{lemma}
The fundamental heap $h(L)$ is well-defined.
\end{lemma}
\begin{proof}
First we observe that the relations do not depend on  the choice of directions of under-arc.
Suppose we traverse the left under-arc of Figure~\ref{crossing} (B) from $z$ to $x$, instead. 
Then the letters of over-arcs that one encounters in this direction is $(v, u)$, in this order.
Hence the convention of defining the relation for this arc is $x=z v^{-1} u$. This is equivalent to the 
relation $z=x u^{-1} v $ defined from the original choice of direction from $x$ to $z$.
The other under-arc is similar.

Second, we show that groups defined are isomorphic under the cancelation move depicted in 
Figure~\ref{cancel}.
Computing from the left, the middle parallel arcs (bottom and top arcs, respectively), are labeled by $xy^{-1}x $ and $x$. Similarly from the right we obtain
${x'}  {y'}^{-1}  {x'} $ and ${x'} $. 
Hence we have ${x'} =x$ and ${y'} =y$ as desired.

It is well known
	\cite{FR}
that the framed link diagrams with blackboard framing are related by sequences of Reidemeister moves of type II, III, and the cancelation move in 
Figure~\ref{cancel}. 
By checking isomorphisms before and after Reidemeister type II and III moves in  a routine manner, we find that the fundamental heap is well defined up to isomorphism (as a group, and hence as a heap).
 The invariance under  
 type III move is indicated in Figure~\ref{heaptypeIII}. 
\end{proof}

\begin{figure}[htb]
\begin{center}
\includegraphics[width=3.8in]{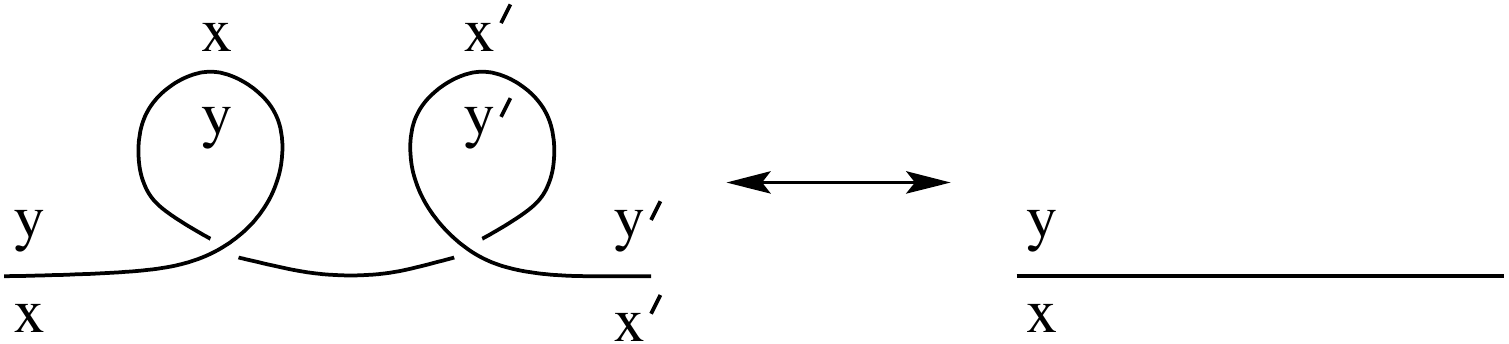}
\end{center}
\caption{Cancelation of a pair of crossings}
\label{cancel}
\end{figure}

Recall  the  telephone cord with writhe $n$, $C_n$, represented by  a diagram depicted in Figure~\ref{telecord}, consisting of small $n \in \Z$ ($n>0$) kinks with $n$ positive crossings. 
Recall also  that either choice of orientations gives rise to positive crossings.  
A negative integer $n$ represents negative  crossings. 
	As indicated in the figure, from the leftmost labels $(x,y)$, after the first crossing,
the labels are computed under the relation to be $(y, yx^{-1}y)$. We reduce the generators 
by applying the relations inductively, and obtain the following.

\begin{lemma}\label{lem:telecord}
Let $C_n$ be a  telephone cord with writhe $n$. Let $x$, $y$ be the two generators assigned at the bottom and top parallel leftmost arcs as indicated in Figure~\ref{telecord}. Then the 
	labels 
at the rightmost bottom and top arcs are 
	labeled
by $x (x^{-1} y  )^n $ and $y ( x^{-1}y )^{n}  $, respectively.

\end{lemma}

\begin{proposition}
Let $\hat{C}_n$, $n \in \Z$,   be the unknot with $n$  kinks (the closure of $C_n$ by a trivial arc). 
Then the  fundamental heap $h(\hat{C}_n)$ is isomorphic to the free product
$\Z * \Z_n$.
\end{proposition}

\begin{proof}
By Lemma~\ref{lem:telecord} we obtain the presentation
for $h(\hat{C}_n)$ to be 
$\langle \ x, y \ | \ (x^{-1} y )^n \ \rangle$.
Set $\alpha=x^{-1 } y  $. Then $y=  x \alpha$ and  the presentation can be rewritten as
$\langle \ x, \alpha \ | \ {\alpha}^n \ \rangle$, and the result follows.
The negative case is similar.
\end{proof}

For torus knots and links $T(2, n)$, we compute the following, with details delayed to Appendix~\ref{app:fund}.

\begin{lemma}\label{lem:torus2n}
	Let $\sigma_1$ indicate the standard 
	generator of the 
	2-string braid group $\mathbb B_2$
	represented by Figure~\ref{crossing} (left) by a ribbon diagram.
Let the top left (doubled) arcs be labeled  by a pair  $(x,y)$ and the right arcs be labeled by $(u,v)$.
Denote by $\sigma_1^n((x,y)\times (u,v)) $ the labels assined to the bottom double arcs of 
a diagram representing the braid word $\sigma_1^n$.  
Set $\alpha := x^{-1}y$ and $\beta := u^{-1}v$. Then we have the following:
	\begin{eqnarray*}
		\lefteqn{\sigma_1^n((x,y)\times (u,v))}\\
	&=&\begin{cases}
		(x\alpha^{-k}(\alpha\beta)^{k}, y\alpha^{-k}(\alpha\beta)^{k})\times (u\beta^{-k}(\alpha\beta)^k,v\beta^{-k}(\alpha\beta)^k) 
		\ \ {\rm if } \ \ n = 2k\\
		(u\beta^{-k}(\alpha\beta)^k,v\beta^{-k}(\alpha\beta)^k)\times (x\alpha^{-(k+1)} (\alpha\beta)^{k+1}, y\alpha^{-(k+1)} (\alpha\beta)^{k+1}) 
		\ \ {\rm if } \ \  n = 2k+1 .
	\end{cases}
	\end{eqnarray*}
	\end{lemma}

Let $D_k^e$ be the group defined by the presentation
$\langle \ \alpha , \beta  \ | \ \alpha^{k} =\beta^{k} =(\alpha\beta)^k   \ \rangle$.
Let  $D^o_k$ be defined by 
$\langle \ \alpha, \beta \ | \ \alpha^{-(k+1)} (\alpha\beta)^{k+1}  \beta^{-k} (\alpha\beta)^k  \ \rangle$. 

\begin{proposition}\label{prop:toruslink}
Let $T(2, n)$ be the framed torus knot or link of  type $(2, n)$ with minimum crossing, that is, the closure of the 2-braid $\sigma_1^n$ that has
$n$ crossings.
Let $F_m$ denote the free group of rank $m$. Then its fundamental heap $h(T(2,n) ) $ is given by:
$$
\begin{cases}
F_2 * D^e_k & {\rm if}  \quad n=2k , \\
F_1 * D^o_k &  {\rm if}  \quad n=2k +1.
\end{cases}
$$
\end{proposition}

\begin{proof} We consider the diagram of $\sigma_1^n$.
Let $n=2k$ be even.
By Lemma~\ref{lem:torus2n}, if the top left and right arcs, respectively, are labeled by $(x,y)\times (u,v)$,  and setting $\alpha=x^{-1}y$ and $\beta=u^{-1}v$,  the bottom left and right arcs are  labeled 
by 
$(x \alpha^{-k} (\alpha \beta)^k , y \alpha^{-k} (\alpha \beta)^k  ) \times (u \beta^{-k} (\alpha \beta)^k , v \beta^{-k} (\alpha \beta)^k  )$. 
By equating top and bottom 
labels 
we obtain
$ x=x \alpha^{-k} (\alpha \beta )^k $, $y = y \alpha^{-k} (\alpha \beta )^k  $, $u=u  \beta^{-k} (\alpha \beta)^k $,  and $v=v \beta^{-k} (\alpha \beta)^k $. Hence we obtain relations  $\alpha^{-k} (\alpha \beta)^k =1$ and $ \beta^{-k} (\alpha \beta)^k =1$.
Thus $h(T(2, n))$ is generated by $x,y,u,v$ with relators $\alpha^{-k} (\alpha \beta )^k $ and $ \beta^{-k} (\alpha \beta)^k $. By adding generators $\alpha, \beta$ and relations $y=x\alpha$ and $v=u\beta$, we obtain
$h(T(2, n))=\langle \ x,u ,\alpha , \beta \ | \ \alpha^{k} =(\alpha \beta)^k = \beta^{k} \ \rangle . $

Similarly, for $n=2k+1$, the bottom left and right arcs are labeled by 
$(u  \beta ^{-k} (\alpha \beta )^k , v  \beta ^{-k} (\alpha \beta )^k  )$ and $(x  \alpha ^{-k-1} (\alpha \beta )^{k+1} , y  \alpha ^{-k-1} (\alpha \beta )^{k+1}  )$, and we obtain relations
$x =u  \beta ^{-k} (\alpha \beta )^k $, $y =v  \beta ^{-k} (\alpha \beta )^k $, $u =x  \alpha ^{-k-1} (\alpha \beta )^{k+1} $, 
and $v=y  \alpha ^{-k-1} (\alpha \beta )^{k+1}$.
By substituting the latter two into the first two, we obtain
$x = (x  \alpha ^{-k-1} (\alpha \beta )^{k+1} ) ( \beta ^{-k} (\alpha \beta )^k) $ and 
$y =( y \alpha ^{-k-1} (\alpha \beta )^{k+1}) (  \beta ^{-k} (\alpha \beta )^k )$.
Hence we obtain 
$$h(T(2, n))=\langle \ x , \alpha , \beta \ | \   \alpha ^{-k-1} (\alpha \beta )^{k+1}   \beta ^{-k} (\alpha \beta)^k  \ \rangle ,$$
and
 the result follows.
\end{proof}

\begin{example}\label{ex:small}
{\rm
In Proposition~\ref{prop:toruslink}, the case $n=2$ is the zero framed Hopf link. In this case
 $D^e_1$ is trivial, and we have $h(T(2,2))=F_2$.
 For $n=4$,  $D^e_2=\langle \ \alpha , \beta  \ | \ \alpha^2 =(\alpha \beta)^2 = \beta^{2}  \ \rangle$,
 which is an infinite group whose abelianization is 
 $\Z_2 \times \Z_2$.

For $n=3$, $D^o_1=\langle \ \alpha, \beta \ | \ \alpha^{-2} (\alpha \beta)^{2}  \beta^{-1} (\alpha\beta)  \ \rangle$.
Set $ \gamma=\alpha\beta \alpha$, then $\beta=\alpha^{-1} \gamma \alpha^{-1}$ can be eliminated and the relation becomes $\gamma^2 \alpha=1$,
so that $D^o_1=F_1 (=\langle \ \alpha , \gamma \ | \ \gamma^2 \alpha \ \rangle)$. Hence $h(T(2,3))=F_1 * F_1 = F_2$.

}
\end{example}

\begin{example} \label{ex:toruscolor}
{\rm
A presentation for  the fundamental heap for diagrams with a different framing for each component 
 can be obtained by modifying relations according  to added kinks  in Figure~\ref{telecord}.

	For example let us consider the torus link on two strands, with $2k$ crossings and framings $n$ and $m$, which we denote by $T_{(n,m)}(2,2k)$. Let $(x,y)$ and $(u,v)$ be the pairs of generators  of the upper (double) arcs of $T_{(n,m)}(2,2k)$.
	Then, applying Lemma~\ref{lem:telecord} the framings change words after kinks  to 
	$$
	( x ( x^{-1} y )^{n}, y(x^{-1} y)^n)\times ( u (u^{-1} v )^{m} , v (u^{-1} v )^m)
	= 
	( x \alpha^{n}, y \alpha^n) \times ( u \beta^{m}, v \beta^m)
	.
	$$
	Observe that $(x  \alpha^{n})^{-1}y \alpha^n =  \alpha$ and 
	$(u \beta^{m})^{-1} v \beta^m = \beta$, following the same notation as in Lemma~\ref{lem:torus2n}, which we now apply to obtain 
	$$
	(x \alpha^{n-k}(\alpha\beta)^k, y \alpha^{n-k} (\alpha\beta)^k)\times ( u  \beta^{m-k}(\alpha\beta)^k, v \beta^{m-k} (\alpha\beta)^k) .
	$$
	We therefore obtain relations 
	(compare with Proposition~\ref{prop:toruslink}):
	$$
x=x \alpha^{n-k}(\alpha\beta)^k,  \ 
y= y \alpha^{n-k} (\alpha\beta)^k ,
\ 
u= u  \beta^{m-k}(\alpha\beta)^k, \
v=v \beta^{m-k} (\alpha\beta)^k .
	$$
From $\alpha=x^{-1}y$ we can replace generators $x,y$ with $x, \alpha$ and similarly for $u, v$, 
so that we obtain 
$$ 
h(T_{(n,m)}(2, 2k) )=
 F_2 (x, u)* \langle \  \alpha, \beta \ | \ 
 \alpha^{n-k}(\alpha\beta)^k,
  \beta^{m-k}(\alpha\beta)^k
 \ \rangle .
$$
We note that 
the relators encode the 
framing of $T_{(n,m)}(2,2k)$. 
}
\end{example}

For  
observing variety of the fundamental heaps, we examine pretzel knots and links.
Recall that the pretzel knot or link $P(n_1, n_2, \ldots, n_r)$  ($r>1$) is given by  $2$-braids $b_i := \sigma_1^{n_i}$, $i=1, \ldots, r$,  where at the top and bottom,  the left arcs of the braid $b_i$ are connected with the right arcs of $b_{i-1}$, and the right ones are conneceted to the left arcs of $b_{i+1}$, identifying the subscripts modulo $r$, i.e. $b_1 = b_{r+1}$. See Figure~\ref{pretzel}.
Proofs are found in Appendix~\ref{app:fund}.

\begin{figure}[htb]
\begin{center}
\includegraphics[width=2in]{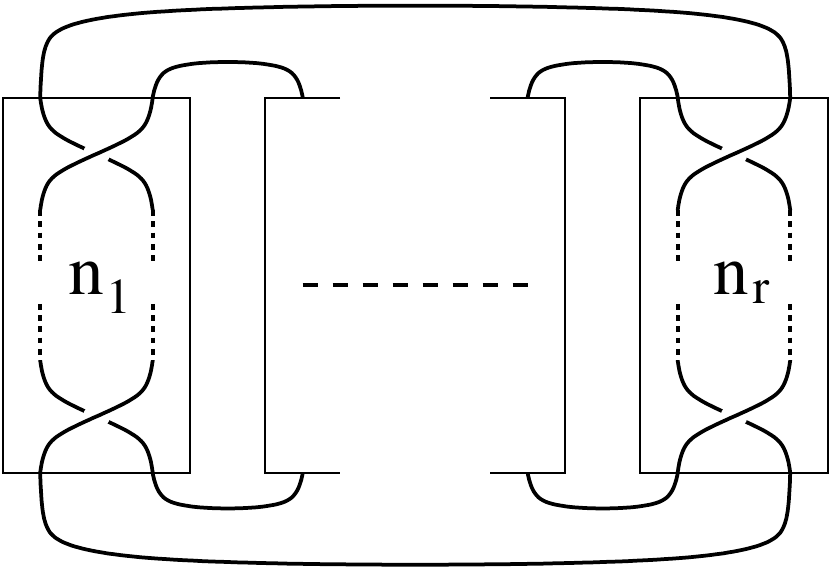}
\end{center}
\caption{Pretzel knots and links}
\label{pretzel}
\end{figure}

\begin{proposition}\label{thm:pretzelnoframe}
Let $P(n_1, n_2, \ldots ,n_r )$  denote the pretzel knot or link with $n_i$ crossings, $i=1, \ldots, r$.
If $n_i=2k_i$ are all even ($i=1, \ldots, r$), then 
 the fundamental heap 
 is described as $h (P(2k_1, \ldots, 2k_r) ) = F_r * G_r$ 
  where $F_r$ is the free group of rank $r$, and $G_r$ has the presentation 
	$\langle \  \alpha_i\mid \Theta_i, i  = 1,\dots, r \  \rangle $
with 
 $$\Theta_i=  \alpha_{i+1}^{-k_i+k_{i+1} }  (\alpha_i \alpha_{i+1}^{-1} )^{k_i}     (\alpha_{i+1} \alpha_{i+2}^{-1} )^{- k_{i+1}} , $$
 where the subscripts are intended to be modulo $r$, so that $\alpha_{r+1} = \alpha_1$. 
	\end{proposition}

In the above examples,  fundamental heaps are free products of free groups and some other groups.
We show that this is the case in general.

\begin{theorem}\label{thm:Wirt}
For any framed link $L$ of $r$ components, its fundamental heap is the free product $h(L)\cong F_r * \hat{h}(L)$ for some group $\hat{h}(L)$ 
generated by $r$ elements, and 
where
$F_r$ is the free group of rank $r$. 
\end{theorem}
\begin{proof}
Every framed link arises as the closure $\hat{\sigma}$ of an $n$-braid $\sigma\in \mathbb B_n$,
where $\mathbb B_n$ denotes the $n$-string braid group, and blackboard framing is assumed with the braid diagrams and their closures. We take the closure to be performed with parallel arcs with blackboard framing and without small kinks. 
In particular, varying framings can be realized by Markov stabilizations.
Let $\sigma_i \in \mathbb B_n$, $i=1, \ldots, n-1$,  be the standard generators that represent positive crossings between $i^{\rm th}$ and $(i+1)^{\rm st}$ strings. Below we regard  framed braid diagrams as double stranded as before. 

Let  $(x_1,y_1)\times \cdots\times (x_n,y_n)$ be the generators assigned  at the  top strings of $\sigma \in \mathbb B_n$.
Let $\alpha_i := x_i^{-1}y_i$ as before, and denote
 $\hat{\sigma}$ the permutation corresponding to $\sigma$.
Then we claim that the labels at the bottom of $\sigma$ is of the form
$\times_{i =1}^{n} (z_i(\sigma), w_i(\sigma))$ written as  $z_i(\sigma) = x_{\hat {\sigma}^{-1}(i)} u_i (\sigma)$ and 
 $w_i(\sigma)= y_{\hat {\sigma}^{-1}(i)} v_i (\sigma) $, where 
 $ u_i (\sigma) $ and  $ v_i (\sigma)$ are words  in $\alpha_i$s and their inverses. 
To show this by induction, it is sufficient to verify this claim for $\sigma'=\sigma \sigma_j$ and 
$\sigma' = \sigma \sigma_j^{-1}$ under the assumption for $\sigma$. 
We have  $z_d(\sigma')=z_d(\sigma)$ and $w_d(\sigma')=w_d(\sigma)$
for $d \neq j, j+1$. 
For $d=j$ we have 
$z_j(\sigma')=z_{j+1}(\sigma)$ and $w_j(\sigma')=w_{j+1}(\sigma)$, 
and for 
$d=j+1$ we have 
\begin{eqnarray*}
 z_{j+1}(\sigma')&=&
z_j(\sigma) z_{j+1} (\sigma)^{-1} w_{j+1} (\sigma) \\
&=&  [  x_{\hat {\sigma}^{-1}(j)} u_j (\sigma) ] \ [  x_{\hat {\sigma}^{-1}(j+1) } u_{j+1} (\sigma)]^{-1}  \
[ y_{\hat {\sigma}^{-1}(j+1)} v_{j+1} (\sigma) ] =x_{\hat {\sigma}^{-1}(j)} u_{j+1} (\sigma')
\end{eqnarray*}
where  $\hat {\sigma}^{-1}(j+1) = \hat{\sigma}_{j} \hat {\sigma}^{-1}(j)$ 
	implies $ \hat {\sigma}^{-1}(j) =  \hat{\sigma}_{j}^{-1}\hat {\sigma}^{-1}(j+1) =
\hat {\sigma '}^{-1}(j+1)$ as desired, 
and 
$$ u_{j+1}(\sigma')=u_j (\sigma)  u_{j+1} (\sigma)^{-1}   { x_{\hat {\sigma}^{-1}(j+1) }}^{-1} 
y_{\hat {\sigma}^{-1}(j+1)} v_{j+1} (\sigma)$$ is indded a word in $\alpha_i$s.
A similar computation can be performed for $w_{j+1}(\sigma')$. 

Thus the relations are of the form
$(x_i, y_i)=( x_{\hat {\sigma}^{-1}(i)} u_i (\sigma),  y_{\hat {\sigma}^{-1}(i)} v_i (\sigma))$, $i=1, \ldots, n$, where $ u_i (\sigma) $ and $v_i (\sigma)$ are words in $\alpha_i$s. 
Let $(x_1, x_{\hat{\sigma} (1)},  x_{\hat{\sigma}^2 (1)}, \ldots, x_{\hat{\sigma}^k (1)} )$
be the orbit of $x_1$ under repeated application of $\sigma$, that constitute a component of $L$ after the closure of $\sigma$. Then we have a set of relations of the form
$x_{\hat{\sigma} (1)}= x_1 u_{(1)}$,  $x_{\hat{\sigma}^2 (1)}=x_{\hat{\sigma} (1)} u_{(2)}, \ldots, x_{\hat{\sigma}^k (1)} =  x_{\hat{\sigma}^{k-1} (1)} u_{(k-1)}$, $x_1= x_{\hat{\sigma}^{k} (1)} u_{(k)}$, where $u_{(j)}$ are words in $\alpha_i$s. Hence this set of relations gives  the generators $x_{\hat {\sigma}^j (1)}$  expressed by $x_1$ and words in $\alpha_i$s, and gives rise to a relation of the form 
$x_1=x_1 u_{(1)} \cdots u_{(k)} $. Similar calculations can be performed to other orbits and $y_i$s.
Recall that $y_i=x_i \alpha_i$ are used to express $y_i$ by means of $x_i$ and $\alpha_i$. 
Hence the fundamental heap has a presentation
$$h(L)=\langle \ x_{j(1)}, \ldots x_{j(r)}, \alpha_i \mid 
R_i, \ i=1, \ldots, r \ \rangle , $$
where  $x_{j(1)}, \ldots,  x_{j(r)}$ are representatives of orbits, and $R_i$ are relations among $\alpha_i$s. Hence $h(L)\cong F_r * \hat{h}(L)$, where $r$ is the number of components, and $\hat{h}$ is the group presented by generators $\alpha_i$ and relations among them.
\end{proof}

\subsection{Relations to Vinberg and other groups}

We relate the fundamental heap to Vinberg and, in special cases, to Coxeter groups for the examples computed in Section~\ref{sec:hdef}.
Recall that a Vinberg group, i.e. {\it group defined by periodic paired relations} in Vinberg's original paper \cite{Vin}, is defined by $n$ generators $a_1, \ldots , a_n$ and relations $a_i^{k_i} = 1$, $w(a_i,a_j)^{k_{ij}} = 1$, where $k_i, k_{ij} \geq 2$ and $w(x_i,x_j)$ is a cyclically reduced word, which can be assumed to be of type $a_i^{t_1}a_j^{s_1} \cdots a_i^{t_d}a_j^{s_d}$ for positive inteders $0< t_q < k_i$ and $0< s_q < k_j$. We follow the convention that if $k_i, k_{ij} = \infty$, then no relation correponding to $a_i$ or $w(a_i,a_j)$ is imposed. Vinberg groups encompass two important classes fo groups, namely that of Coxeter groups and generalized triangule groups. Due to the role they play in Theorem~\ref{thm:Vin} below, we recall that a Coxeter group  has a presentation 
$ \langle \ a_i \ (i=1, \ldots, n ) \ | \ (a_i a_{j})^{m_{ij}} \ (1 \leq i, j \leq n) \ \rangle  $
where $m_{ii}=1$ and $m_{ij}\geq 2$, where we follow the previous convention on infinite exponents.

\begin{proposition}
Let $h(T(2,n)) = F_2*D_k^e$ be the fundamental heap of $T(2,n)$ for $n=2k$ and 
$  F_1*D_k^o $ for $n=2k+1$ as determined in Proposition~\ref{prop:toruslink}.
Then there exist group epimorphisms from $D_k^e$ and  $D_k^o$ to  Vinberg groups for all $k>0$ and $k\geq 3$, respectively.
\end{proposition}
\begin{proof}
	For $n=2k+1$ and $k\geq 3$, 
	consider the Vinberg group with presentation $V(k):= \langle \ x,y \mid x^k = y^{k-1} = (xy)^k =1\ \rangle$. 
	Then mapping $f$ defined by $\alpha \mapsto x$, $\beta \mapsto y$ we obtain a map from the free group $F_2$ on two elements $\alpha, \beta$  that descends to a well defined homomorphism on $D^o_k$:
			We  verify that 
			the relator of $D^o_k$ maps to the identity in $V(k)$
		for $k\geq 3$:
$f(\alpha^{-(k+1)} (\alpha\beta)^{k+1} \beta^{-k} (\alpha\beta)^k) 
=  x^{-k-1} (xy)^{k+1} y^{-k} (xy)^k 
=  x^{-1} (xy) y^{-1} 
= 1	$ as desired.
	
		Similarly the even case  in Proposition~\ref{prop:toruslink},
$D^e_{k} $ surjects to another Vinberg group 
$\langle \ x,y \ | \ x^k =y^k=(xy)^k=1  \ \rangle$.
\end{proof}

We note that it is known that for the odd case $n=2k+1$, $V(k)$ is infinite. 
For the even case, the group  $D(k,l,m) := \langle \ x,y \ | \ x^k =y^l=(xy)^m=1  \ \rangle$  is also called   an {\it ordinary triangle group}, or 
 a von Dyck group, and 
 it is known that $D(k,l,m)$ is non-abelian and  infinite if $\frac{1}{k} + \frac{1}{l} + \frac{1}{m} \leq 1$, see the discussion preceding Theorem~2 in \cite{Vin}.

For pretzel  links we have the following.

\begin{theorem}\label{thm:Vin}
Let $h(P(2k_1, \ldots, 2k_r))= F_r * G_r$ be the fundamental heap of pretzel link 
as in Proposition~\ref{thm:pretzelnoframe}. 
Then if $|-k_i + k_{i+1}| \neq 1$ for all $i = 1,\dots , r$, there exists a group epimorphism from 
$G_r$ 
to a Vinberg group. If, moreover $k_i$ ($i=1, \ldots, r$) are either all odd or all even, we can specialize the Vinberg group to be Coxeter.
\end{theorem}

\begin{proof}
We apply the even case of Proposition~\ref{thm:pretzelnoframe}. 
Let us consider the Vinberg group $V:= V(P(2k_1, \ldots, 2k_r))$ with generators $a_1, \dots , a_r$ and relations $a_i^{ |-k_{i-1}+k_i|} = 1$ when $|-k_{i-1}+k_i| \neq 0$, $a_i^{\infty} = 1$ when $|-k_{i-1}+k_i| = 0$, and $(a_ia_{i+1}^{|-k_i+k_{i+1}|-1})^{k_i}=1$. 
The exponents $ |-k_i+k_{i+1}|-1$ in the latter are for  making them positive.
Now define the map $f$ on generators 
$\alpha_i \mapsto a_i$ for all $i = 1,\dots , r$. We  verify that $f$ maps the relators $\Theta_i$ to the identity in $V$. For $i$ such that $|-k_{i-1}+k_i| \neq 0$ we have 
\begin{eqnarray*}
f(\Theta_i) &=& f(\alpha_{i+1})^{|-k_i+k_{i+1}|}(f(\alpha_i)f(\alpha_{i+1})^{-1})^{k_i}(f(\alpha_{i+1})f(\alpha_{i+2})^{-1})^{-k_{i+1}}\\
 &=& a_{i+1}^{-k_i+k_{i+1}}(a_ia_{i+1}^{|-k_i+k_{i+1}|-1})^{k_i}(a_{i+1}a_{i+2}^{|-k_{i+1}+k_{i+2}|-1})^{-k_{i+1}} \quad = \quad 1. 
\end{eqnarray*}
If $i$ is such that $|-k_{i-1}+k_i| = 0$ we have
\begin{eqnarray*}
	f(\Theta_i) &=& (f(\alpha_i)f(\alpha_{i+1})^{-1})^{k_i}(f(\alpha_{i+1})f(\alpha_{i+2})^{-1})^{-k_{i+1}}\\
	&=& (a_ia_{i+1}^{-1})^{k_i}(a_{i+1}a_{i+2}^{-1})^{-k_{i+1}}  \quad = \quad 1 .	\end{eqnarray*}
Suppose now $k_i$ ($i=1, \ldots, r$) are either all odd or all even. Set $m_{ij}=k_i$ if $j=i+1$ and $m_{ij}=\infty$ otherwise.
Let $$f: F(\alpha_i \ (i=1, \ldots, r) ) \rightarrow \langle \ a_i \ (i=1, \ldots, n ) \ | \ (a_i a_{j})^{m_{ij}} \ (1 \leq i, j \leq n) \ \rangle $$
be the epimorphism from the free group to the Coxeter group. 
Then it is sufficient to show that the relations $\Theta_i$, $i=1, \ldots, r$, hold in the image.
Since  $k_i$s are all even or all odd, we have 
 $f(\alpha_{i+1} )^{-k_{i}+k_{i+1}} = f(\alpha_i)^{2 \ell_i }=a_i^{2 \ell_i} =1$, where $\ell_i$ is some integer.
Since $f(\alpha_i)^2=a_i^2=1$, we have 
$f(\alpha_i^{-1})=f(\alpha_i)$, and
$f( (\alpha_i \alpha_{i+1}^{-1})^{k_i})= f( (\alpha_i \alpha_{i+1})^{k_i})  =1$. 
Hence $\Theta_i $ is sent to the identity  in the image.
\end{proof}

\begin{corollary}
	If $\sum_{i=1}^r \frac{1}{k_i} \leq r-1$, then the group $G_r$ is infinite.
\end{corollary}

\begin{proof}
Applying Corollary to Theorem~1 and Theorem~2 in \cite{Vin} it follows that $V(P(2k_1, \ldots, 2k_r))$ is infinite if $\sum_{i=1}^r \frac{1}{k_i} \leq r -1$. Using Theorem~\ref{thm:Vin} we have that $G_r$ maps onto $V(P(2k_1, \ldots, 2k_r))$, which completes the proof. 
\end{proof}

\begin{remark}
{\rm 
We note that similar epimorphisms exist 
in the case of pretzel links with 
different 
framings.
Recall first, that pretzel links $P(2k_1,\ldots , 2k_r)$ with all even crossings have $r$ components and, therefore, upon rearranging the twists we can position the framings on the left (pair of) arcs above each sequence of $2k_i$ crossings. We denote the sequence of framings $\bar m = (m_1, \ldots , m_r)$ and indicate the pretzel link with framings $\bar m$ by $P^{\bar m}(2k_1, \ldots , 2k_r)$.

We use Lemma~\ref{lem:telecord} to insert additional framings and Lemma~\ref{lem:torus2n} to insert crossings. The bottom colorings after adding twists then become
$$
(y_i\alpha_i^{m_i-1-k_i}(\alpha_i\alpha_{i+1})^{k_i}, y_i\alpha_i^{m_i-k_i}(\alpha_i\alpha_{i+1})^{k_i})\times (y_{i+1}\alpha_{i+1}^{-k_i}(\alpha_i\alpha_{i+1})^{k_i}, x_{i+1}\alpha_{i+1}^{-k_i}(\alpha_i\alpha_{i+1})^{k_i}).
$$	
We obtain the same relation $\alpha_i^{k_{i+1}-k_1-m_i}(\alpha_i\alpha_{i+1})^{k_i} = (\alpha_{i+1}\alpha_{i+2})^{k_{i+1}}$, 
hence 
$h(P^{\bar m}(2k_1, \ldots , 2k_r)) = F_r* \bar{G}_r$
where $
\bar{G}_r = \langle\  \alpha_i \mid \ \Theta_i, \ i = 1,\ldots, r \ \rangle ,
	$
	with $\Theta_i = \alpha_{i+1}^{k_{i+1}-k_i-m_i} (\alpha_i\alpha^{-1}_{i+1})^{k_i} (\alpha_{i+1}\alpha_{i+2})^{-k_{i+1}}$. 

	As in Theorem~\ref{thm:Vin} define the Vinberg group by presentation 
	$$\langle \ a_i \mid \ a_i^{k_{i+1}-k_i-m_i} = ( a_ia_{i+1} )^{k_{i+2}-k_{i+1}-m_{i+1}-1} = 1,\ i = 1,\ldots , r\  \rangle , $$
	then there is an epimorphism from $\bar{G}_r$ to this Vinberg group.

}\end{remark}

\begin{figure}[htb]
\begin{center}
\includegraphics[width=1in]{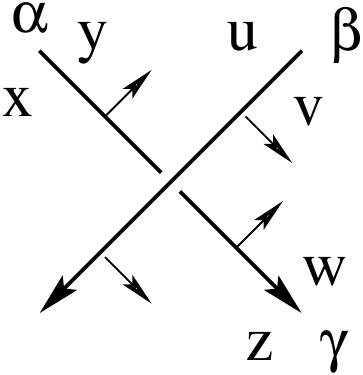}
\end{center}
\caption{Wirtinger presentation}
\label{Wirtinger}
\end{figure}

\subsection{Relations to Wirtinger presentations}\label{sec:Wirt}

We   describe a relation between the fundamental heap of framed links and the Wirtinger presentation of knot (link) groups.
Let $\vec{L}$ be an oriented link, and let $L$ denote the link obtained 
by forgetting the orientation of $\vec{L}$.

Let $\pi(\vec{L})$ denote the link group. 
Then we define a group homomorphism $\lambda: \pi(\vec{L}) \rightarrow h(L)$ as follows,
using Wirtinger presentations of $\pi(\vec{L})$.
Let $D$ be an oriented diagram of $\vec{L}$, with the same notation $D$ as a diagram of $L$ when orientations are forgotten.

First consider the case when the orientations of both arcs in Figure~\ref{crossing} (A) point down.
A single stranded version is depicted in Figure~\ref{Wirtinger} with orientations and orientation normal vectors defined via  the right-hand rule.
Let $\alpha, \beta $ and $\gamma$ be Wirtinger generators assigned to the arcs in the figure, where 
pairs of generators $(x, y)$, $(u,v)$, and $(z, w)$, respectively, are assigned to corresponding double arcs in the figure.
Then the group homomorphism is defined by the assignment of generators $\lambda ( \alpha )=x^{-1} y$,
$\lambda ( \beta)=u^{-1} v $ and $\lambda ( \gamma)=z^{-1}w$, respectively. 
If the orientation of the over-arc labeled by $\beta$ is reversed, then the assignment changes 
to $\lambda (\beta)=v^{-1} u $. Compare with Definition~\ref{def:fund}.

\begin{proposition}\label{prop:Wirt}
The homomorphism 
$\lambda: \pi(\vec{L}) \rightarrow h(L)$ defined above by Wirtinger presentation,
is indeed well-defined.
\end{proposition}
\begin{proof}
This is checked by computing the Wirtinger relation.
When the orientation of the both arcs in Figure~\ref{crossing} (A) point downwards, then 
one computes
$$\lambda( \beta^{-1} \alpha \beta )= (u^{-1} v)^{-1} (x^{-1} y) (u^{-1} v)
= (x u^{-1} v)^{-1} (  y u^{-1} v)=z^{-1} w=\lambda(\gamma).$$
When the orientation of the over-arc is reversed, then using the changed assignment 
$\lambda(\beta)=v^{-1} u$, and the Wirtinger relation for a negative crossing is checked by 
$$
\lambda( \beta \alpha  \beta^{-1}  )= (v^{-1} u) (x^{-1} y) (v^{-1} u)^{-1}
= (x u^{-1} v)^{-1} (  y u^{-1} v)=z^{-1} w=\lambda(\gamma)
$$
as well. 
The other cases follow similarly.
\end{proof}

\begin{remark}
{\rm

From the proof of Theorem~\ref{thm:Wirt}, when the fundamental heap is written as 
$h(L)=F_r * \hat{h}(L)$ for a framed link $L$, we find that $\lambda$ 
surjects to
 $\hat{h}(L)$.
In fact, $\hat{h}(L)$ is generated by $\alpha_i$'s, and these are the images of generators of the Wirtinger presentation of $L$ under $\lambda$. 
}
\end{remark}

\begin{example}
{\rm
For those in Example~\ref{ex:small}, we observe the images of $\lambda$.
 In the case of $n=2$ (the Hopf link),
 $D^e_1$ is trivial, and we have $h(T(2,2))=F_2$, so that ${\rm Im}(\lambda)=1$. 
 For $n=4$, $D^e_2=\langle \ \alpha , \beta  \ | \ \alpha^2 =(\alpha \beta)^2 = \beta^{2}  \ \rangle$ is isomorphic to $\Z_2 \times \Z_2$, and  $h(T(2,4))=F_2 * (\Z_2 \times \Z_2 )$,
 and ${\rm Im}(\lambda)=\Z_2 \times \Z_2 $, so that $\lambda$ factors through the abelianization $\Z \times \Z$ and each factor of $\Z$ surjects to $\Z_2$ from the presentation involving $\alpha$ and $\beta$. 

For $n=3$, for $D^o_1=\langle \ \alpha, \beta \ | \ \alpha^{-2} (\alpha \beta)^{2}  \beta^{-1} (\alpha\beta)  \ \rangle$, 
we set $ \gamma=\alpha\beta \alpha$, and obtained  $D^o_1=\Z (=\langle\  \alpha , \gamma \ | \ \gamma^2 \alpha \ \rangle)$. 
The generator of $\Z$ is $\alpha$ which corresponds to a meridian, so we find that $\lambda$ is abelianization. 

}
\end{example}

\section{Colorings and cocycle invariants of framed links by heaps}\label{sec:color}

A coloring of a framed link diagram by a heap  is defined by assigning elements of the heap to these double arcs as follows, in a manner similar to quandle coloring, and cocycle invariants are also similarly defined as in ~\cite{CJKLS}. In this section we give such definitions, examples, and applications to the rank of heap TSD cohomology.

\subsection{Colorings} 

First we define and examine colorings of framed link diagrams by heaps.

\begin{definition}
{\rm
Let $X$ be a heap. 
Let $D$ be an unoriented framed link ribbon diagram and ${\cal A}$ the set of doubled arcs.
A {\it coloring } of $D$ by $X$ is a map ${\cal C}: {\cal A} \rightarrow X$ that satisfies the {\it coloring condition}
as depicted in Figure~\ref{crossing}  (B), where $(z,w) = (xu^{-1}v, y u^{-1} v )$.
}
\end{definition}

From the definition we obtain the following by checking the moves.

\begin{lemma}\label{lem:color}
The sets of colorings of  two framed link diagrams 
are in bijection between Reidemeister type II, III moves and the cancelation move.
\end{lemma}

In particular, the number of colorings of a framed link diagram by a finite heap $X$ is an invariant of the framed link $L$, that does not depend on the choice of a diagram, and is denoted by ${\rm Col}_X(L)$. 

\begin{remark}
{\rm
The set of colorings of a framed link $L$ by a heap $X$ can be considered as the set of heap homomorphisms
from $h(L)$ to $X$. Although the fundamental heap was defined by group presentations, these homomorphisms need not be group homomorphisms; assigning a single color to all arcs that is not the identity element is a heap homomorphism but not a group homomorphism. 
}
\end{remark}

We observe that from the definition, if $x=y$ at a crossing as in Figure~\ref{crossing}, then 
we have $z=w$. Consequently, if $x=y$ (the two colors are equal) at an arc, then all the arcs of this  component have this property. Conversely, if $x \neq y$ at a crossing as in Figure~\ref{crossing}, then 
the two colors of two strings are distinct at every arc of the component.

\begin{definition}\label{def:monocolor}
{\rm 
For a given coloring of a framed knot diagram by a  heap, a coloring with the same colors at each double  arc 
is called a {\it monocoloring}, and 
a component of a framed link with the property that double arcs receive the same color 
 is called
a {\it monochromatic} component. 

A coloring that is not monochromatic is called {\it bicoloring}, and 
a component of a framed  link with the property that double arcs receive distinct  colors
 is called
 a  {\it bicolored} component. 
 
 For a framed knot $K$, we denote by ${\rm Col}_X^{\rm B}(K)$  the number of bicolorings of $K$ by $X$.
}
\end{definition}

We note that if the over-arc is of monochromatic component, then 
the colors of the under-arc in Figure~\ref{crossing} 
satisfies $x=z$ and $y=w$, i.e., the monocolored over-arc does not change the colors of the under-arc.

From the definition we find that 
for  a framed knot $K$, and  a finite heap $X$, 
 the number of monochromatic colorings of $K$ by $X$ is $|X|$.

\begin{example}\label{ex:cordcolor}
{\rm
Let $\hat{C}_n$ be the closure of the diagram in Figure~\ref{telecord} with the left and the right end points connected with no twists, with the writhe $n$. 
Let $X$ be a group heap.

The set of colorings ${\mathcal C}$ of $\hat{C}_n$ by $X$ is in bijection with the pairs 
$(x, y) \in X \times X$ such that $( x^{-1}y  )^n =1$. The bijection is induced from assigning $(x,y)$ at the leftmost arcs as depicted in Figure~\ref{telecord}. 
In particular, for $X=\Z_n$, there are $n$ monocolorings and 
 $n(n-1)$  bicolorings. For $X=\Z_m$ with $(n,m)$ coprime, there are $m$ monocolorings and no bicoloring.

}
\end{example}

\begin{example}\label{ex:dihedraltorus}
	{\rm 
We examine colorings of $T_{(n,m)}(2,2k)$ when $X$ is the dihedral heap $D_3$, generated by a rotation $r$ and a reflection $a$.
Relations $\alpha^{k-n}=(\alpha \beta)^k =\beta^{k-m}$ in Example~\ref{ex:toruscolor} can also be regarded as coloring conditions. We use the same letters, $x,y,u,v$ and $\alpha=x^{-1}y$, 
$\beta=u^{-1}v$.

\begin{itemize}
\setlength\itemsep{-3pt}
\item (Case 00) If the both components represented by $\alpha$ and $\beta$ are monocolored, then $\alpha=1=\beta$.
Then  there are $6 \times 6=36$ colorings for assignments of elements of $D_3$ for $x=y$ and $u=v$, using letters in 	Example~\ref{ex:toruscolor}.

\item (Case 01) 
$\alpha=1$ (monocolored) and $\beta \neq 1$ (bicolored),
 then $\beta^{k-m}=\beta^k=1$,
that is equivalent to $\beta^m=\beta^k=1$. For each case $x=y$ has 6 choices. 
\begin{itemize}
\setlength\itemsep{-3pt}
\item (O2) The order of  $\beta$ is 2:  Both $m$ and $k$ need to be both even,
 $\beta=a, ra, r^2a $ and each case has 6 choices of $u$. There are $6$ (for $x$) $\times 3$ (for $\beta$) $\times 6$ (for $u$) colorings. 
 
\item (O3)  The order of  $\beta$ is 3:  Both $m$ and $k$ need to be  divisible by 3,
 $\beta= r, r^2 $ and each case has 6 choices of $u$. 
 There are $6$ (for $x$) $\times 2$ (for $\beta$) $\times 6$ (for $u$) colorings. 
\end{itemize}	
\item (Case 10) 
 $\beta=1$ (monocolored) and $\alpha \neq 1$ (bicolored), similar to (Case 01).
\item (Case 11)  $\alpha\neq 1$ and $\beta \neq 1$ (both bicolored). Again each case below has 6 choices each for $x$ and $u$.
\begin{itemize}
\setlength\itemsep{-3pt}
\item (O22)  The orders of  $\alpha$ and $\beta$ are both 2:
Both $n$ and $m$ need to be both even. 

\begin{itemize}
\item
 If $\alpha=\beta$, then $\alpha\beta=1$ and $k$ has no constraint.
There are $6$ (for $x$) $\times 3$ (for $\alpha =\beta$) $\times 6$ (for $u$) colorings. 

\item If $\alpha\neq \beta$, then the order of $\alpha\beta$ is 3, so that  $k$ needs to be divisible by $3$,
in which case there are $6$ (for $x$) $\times (3 \times 2)$ (for $\alpha \neq \beta$) $\times 6$ (for $u$) colorings. 
\end{itemize}

\item (O23) The order of  $\alpha$ is 2 and the order of $\beta$ is 3:
 $n$ needs to be even and $m$ need to be divisible by $3$. Since $\alpha=a, ra, r^2 a$ and $\beta=r, r^2$, the order of $\alpha \beta$ is 2, so that $k$ needs to be even, in which case
there are $6$ (for $x$) $\times 3$ (for $\alpha$)  $\times 2$ (for $\beta$) $\times 6$ (for $u$) colorings. 

\item (O32) The order of  $\alpha$ is 3 and the order of $\beta$ is 2:
Similar to the above.

\item (O33) The orders of  $\alpha$ and $\beta$ are both 3:
Both $n$ and $m$ need to be both divisible by 3. 

\begin{itemize}
\item
 If $\alpha=\beta=r, r^2$, then $\alpha\beta$ 
has order 3 and $k$ needs to be divisible by 3,
in which case there are  $6$ (for $x$) $\times 2$ (for $\alpha =\beta$) $\times 6$ (for $u$) colorings. 

\item If $\alpha\neq \beta$, then $\alpha\beta=1$
and $k$ has no constraint, and 
there are  $6$ (for $x$) $\times 2$ (for $\alpha \neq \beta$) $\times 6$ (for $u$) colorings. 
\end{itemize}

\end{itemize}
\end{itemize}

We note that from Section~\ref{sec:Wirt}, assignments of $\alpha$ and $\beta$ correspond to Wirtinger relation. In this case, assignments of reflections corresponds to Fox tricoloring, which is 
Case 11 (O22). Further, Subcase (A) corresponds to the trivial Fox colorings, and (B) non-trivial.

We also note that there are cases for arbitrary choice of $x$ and $u$. This corresponds to the free group factor $F_r$ in Theorem~\ref{thm:Wirt}.

	}
\end{example}

\subsection{Cocycle invariants}\label{sec:cocyinv}

In this subsection  we consider heap cocycle invariants of framed links.
Let $X$ be a heap  and $D$ an oriented blackboard framed link.
Then a 2-cocycle invariant is defined in a manner similar to the quandle 2-cocycle invariant
as follows.

For defining the invariant, we consider oriented framed link diagrams. An  (orientation) normal (vector)  is assigned at each oriented arc  of a diagram as  in  Figure~\ref{Wirtinger}, in which the crossing  is positive, and  otherwise negative. The {\it sign} $\epsilon(\tau)$ of a crossing $\tau$ is defined to be $1$ (resp. $-1$) if $\tau$ is positive (resp. negative). 
The orientation normal of the over-arc labeled $(u,v) $ points  from the arc labeled $(x,y)$ to the arc labeled $(z,w)$.

\begin{definition}\label{def:cocyinvariant}
{\rm
Let $X$ be a heap, $A$ an abelian group 
with multiplicative notation, 
and  $\psi$ be a 2-cocycle, $\psi \in Z^2_{SD}(X,A)$. 
Let $D$ be an oriented diagram of a knot $K$ with blackboard framing. 
In Figure~\ref{crossing} (right), assume that the over-arc is oriented downward. 
The {\it Boltzmann weight} $B_\ell({\mathcal C}, \tau)$, for $\ell=0,1$,  for a coloring ${\mathcal C}$ of a crossing $\tau$ 
is defined by $\psi(x_\tau, u_\tau, v_\tau)^{\epsilon(\tau) }$  and $\psi(y_\tau, u_\tau, v_\tau)^{\epsilon(\tau) }$, respectively, where 
colors at $\tau$ are  as depicted in Figure~\ref{crossing} with $\tau$ indicating the specified crossing.

The {\it 2-cocycle heap invariant} of a heap $X$ with respect to a 2-cocycle $\psi \in Z^2_{SD}(X,A)$
is defined by 
$$ \Psi_\psi (K) = \sum_{\mathcal C}
(  \prod_{\tau}  B_0 ({\mathcal C}, \tau),   \prod_{\tau}  B_1({\mathcal C}, \tau)), $$
which takes values in $\Z[A \times A]\ ( \cong \Z[A] \otimes \Z[A] ) $. 
}
\end{definition}

	The assignment of a 2-cocycle at a crossing is depicted in Figure~\ref{heaptypeIII} at the top left crossing.
The following is proved by arguments similar to those found in \cite{CJKLS}. We provide an outline of proof.

\begin{theorem}\label{thm:cocyinvariant}
	The $2$-cocycle heap invariant is indeed an invariant of framed links. Moreover, a 2-coboundary yields the trivial invariant value (in the $\Z \otimes \Z$ factor  spanned by $e \otimes e$ for the group identity $e$). 
	The cocycle invariant  $\Psi_\psi (K)$ depends only on the cohomology class $[\psi]\in H^2_{\rm SD}(X,A)$. 
\end{theorem}
\begin{proof} 
	By Lemma~\ref{lem:color}, the sets of colorings by a group heap are in bijection between
Reidemeister moves of types II, III, and the cancelation moves.
Hence the invariance of the product of weights is checked between the moves.
The invariance under the Reidemeister type II move and the cancelation move follows from 
 the sign convention of $\epsilon(\tau)$ in the exponent of the weight. 
The invariance under the type III move follows from the 2-cocycle condition. 
	
	The cocycle invariant $\Psi_\psi(L)$ for a framed link has the following interpretation as the Kronecker product as for the quandle cocycle invariant \cite{CJKLS}. 
Let $K=K_1 \cup K_2$ be a blackboard framed double stranded colored diagram of a component of  the link $L$.
Then colored components $K_i$ represents 2-cycles of  $Z_2^{\rm SD} (X, \Z)$ for $i=1,2$.
If $K_i$, $i=1,2$, 
 contribute $g_1 \otimes g_2$ in $\Z[A] \otimes \Z[A] \cong \Z [ A \times A]$ to $\Psi_\psi(L)$, 
where $g_1 , g_2 \in A$, then this means that $g_i =\braket{ \psi \mid K_i }=\psi(K_i)$ (the Kronecker product) for $i=1,2$. This holds for all double stranded components of all components of $L$.
If   $\psi  $ is a coboundary, $\psi= \delta f$
then  
$$\braket{ \ \sum_i a_i \psi_i \mid C\  }= \braket{\  \delta f \mid C\  }=\braket{ \ f  \mid \partial C \ } = 0  $$ for all 2-cycles $C$.
Hence any coloring contributes $e \otimes e$ to the invariant.
\end{proof}

Considering the decomposition of cohomology into
degenerate and nondegenerate parts in Proposition~\ref{prop:DH}, we characterize the cocycle invariant for degenerate cocycles as follows.

\begin{lemma}\label{lem:DH}
Let $X$ be a heap, $A=\Z_n$,  and $\psi \in Z^2_{\rm DH}(X, A)$ a generating degenerate 2-cocycle in Proposition~\ref{prop:degcohomology}.
Let $e\in A$ be the identity element,
and $g$ a multiplicative generator of $A$ such that $\psi(x,x,x)=g$. 

Let $K$ be a framed knot with writhe $n$. Then the cocycle invariant is 
$$\Psi_\psi (K)= {\rm Col}_X^{\rm B} (K) ( e \otimes e) + |X| (g^n \otimes g^n) , $$
where ${\rm Col}_X^{\rm B} (K) $ denotes the number of bicolorings in Defition~\ref{def:monocolor}. 
\end{lemma}

\begin{proof}
Each bicoloring contributes $e \times e \in A \times A$ to the Boltzmann weight since such a coloring evaluates $e$ by a degenerate cocycle $\psi$. Hence bicolorings contribute the first term
${\rm Col}_X^{\rm B} (K) ( e \otimes e) $.

For each monochromatic coloring, a crossing in Figure~\ref{crossing} (B), with labels regarded as colors,
contribute $g^{\epsilon(\tau)} \times g^{\epsilon(\tau)} $ to the Boltzmann weight, where 
$\epsilon(\tau)$ denotes the sign of the crossing. The number of monochromatic colorings is $|X|$, so that they contribute the term  $|X| (g^n \otimes g^n)$.
\end{proof}

\begin{example}\label{ex:telecordinv}
{\rm

Let $X$ be a group heap, $A$ an abelian group, and $\psi \in Z^2_{\rm SD}(X,A)$ a 2-cocycle. 
Let $\hat{C}_n$ be the closure of the diagram in Figure~\ref{telecord} with the left and the right end points connected with no twists, with the framing $n>0$. 
From the figure, with colors $(x , y )$ at the leftmost arcs, we see that the color changes from 
$(x , y )$ to $(y , y  x^{-1} y )$ at the first crossing, and from 
$( x( x^{-1} y )^{i-1}  , y( x^{-1}y )^{i-1}  )$ to $(  x ( x^{-1} y  )^{i}  , y ( x^{-1} y )^{i} )$ at the $i^{\rm th}$-crossing for $i>1$.
Hence the 2-cocycle invariant of $\hat{C}_n$ with $\psi$ is expressed as 
\begin{eqnarray*}
\Psi_{\psi} (\hat{C}_n)&=&
\sum_{ {\mathcal C} } \ 
(  \ 
 \prod_{i=1}^{n}  \psi( (x (x^{-1} y )^{i-1} ,  x ( x^{-1} y)^{i} , y( x^{-1}y )^{i}  ) ,  \\
& & \hspace{10mm} 
 \prod_{i=1}^{n} 
 \psi ( y (  x^{-1}y )^{i-1}  , x (x^{-1} y)^{i}  ,y ( x^{-1} y)^{i}  ) \  )   .
\end{eqnarray*}

	From Example~\ref{ex:cordcolor} and Lemma~\ref{lem:DH}, we consider the case $A=\Z_n$ and with a 
nondegenerate 2-cocycle  in Lemma~\ref{lem:2cocy},
$\psi(x, y,z)=\psi_{(1,0,0)}(x, y,z)=x(z-y)$.

The first weight $ 
\prod_{i=1}^{n}  \psi( (x (x^{-1} y )^{i-1} ,  x ( x^{-1} y)^{i} , y( x^{-1}y )^{i}  )
 $
 is computed in additive notation, by setting $\alpha=y-x$, as 
$$	 \sum_{i=1}^n ( [  ( x+  (i-1) \alpha  ) ( ( y+ i\alpha  ) - ( x + i \alpha  ) ] 
= \sum_{i=1}^{n} [x+ (i-1)\alpha ]  \alpha  =
		[ \sum_{i=1}^{n} (i-1) ]  \alpha^2 , 
$$
where the last summation is 0 for odd $n$ and $n/2$ for even $n$.
Similarly the second weight $
	 \prod_{i=1}^{n} 
 \psi ( y (  x^{-1}y )^{i-1}  , x (x^{-1} y)^{i}  ,y ( x^{-1} y)^{i}  )
 $ is
$$	 \sum_{i=1}^n ( [ y+   (i-1) \alpha  ) ( ( y+  i\alpha  ) - ( x+ i \alpha) ) ] 
	=  \sum_{i=1}^{n} [ y + (i-1)\alpha ]\alpha
	=
		[ \sum_{i=1}^{n} (i-1) ] \alpha^2   .
$$
Then the weight product   is $e  \otimes e $ ($e$ is a multiplicative generator of the cyclic group, that is the additive group of $\Z_n$)
for odd $n$ and $  (n/2)\alpha^2  \otimes  (n/2)\alpha^2 $ for even $n$.
Hence  if $n$ is odd, then  $\Psi_\psi (\hat{C}_n) = n^2 ( e \otimes e ) $.
If $n$ is even, then 
  $\Psi_\psi (\hat{C}_n) = n( e \otimes e )+ n [\,   \sum_{0 \neq \alpha \in \Z_n}  ( (n/2)\alpha^2  \otimes  (n/2)\alpha^2)   \, ] $,
where the first term corresponds to colorings with $x =y$ and the second to those with $x  \neq y$.
In the second sum, for each  $\alpha=y-x\neq 0$, there are $n$ choices of $x$, hence the coefficient $n$.
The non-triviality of the invariant in the case of $n$ even corresponds to the non-triviality in Proposition~\ref{prop:nondeg}. In particular, the bicolored diagram of $\hat{C}_n$ for $n$ even 
represents a non-trivial class of $H_2^{\rm NDH}(\Z_n, \Z_n)$. 
}
\end{example}

We define the 2-cocycle invariant for framed links componentwise as in   \cite{CJKLS} as follows.

\begin{definition}\label{def:linkinv}
{\rm 
Let $L=K_1 \cup \cdots \cup K_\mu$ be an oriented link diagram of $\mu$ components. 
A crossing   $\tau =\tau(j)$ of $L$ is said to be of the component  $j$ ($j=1, \ldots, \mu$) if the under-arc belong to the component $K_j$. 
Let $X$ be a group heap 
and $\psi \in Z^2_{\rm SD}(X, A)$ be a 2-cocycle with coefficient abelian groups $A$. 
For a coloring ${\cal C}$, the weight $B_\ell^{(j)}({\cal C}, \tau)$ for $\tau=\tau^{(j)}$ of the component $j$ is defined to be $ \psi(x_\ell, y_0, y_1)^{\epsilon(\tau)}  \in A $ where $\epsilon(\tau)$ is the sign of $\tau$ and $\ell=0,1$.
Then the cocycle invariant is defined by 
$$ \vec{\Psi}_{\psi} (L)= \sum_{ \cal C} ( \  ( \prod_{ \tau^{(1)} } B_0^{(1)} ( {\cal C}, \tau^{ (1)} ), 
\prod_{ \tau^{(1)} } B_1^{(1)} ( {\cal C}, \tau^{ (1)} ) )  , \ldots, 
 ( \prod_{ \tau^{(\mu)} } B_0^{(\mu)} ( {\cal C}, \tau^{ (\mu)} ), 
\prod_{ \tau^{(\mu)} } B_1^{(\mu)} ( {\cal C}, \tau^{ (\mu)} )
\  ) 
$$
 where the sum is the formal sum on each component of vectors. 
 }
 \end{definition}

 \begin{example}\label{ex:Zninv}
 {\rm
 We determine the cocycle invariant for the torus link $T(2, 2n)=T_{(0,0)}(2,2n)$ with $2n$ crossings 
using Lemma~\ref{lem:torus2n} with $X=\Z_n=\braket{r}$ for coloring and the $2$-cocycles 
$\phi_i=\sum_{x \in \Z_n } [ \sum_{j=0}^{n-1}  \chi_{(x,r^j,r^{j+i})} ] $, $i=1, \ldots, n-1$,  in Lemma~\ref{lem:Zn}. Let the coefficient group be $A=\Z=\braket{g}$. 

Recall from Lemma~\ref{lem:torus2n}  that when the top  arcs are colored by $(x, y)$ (left top) and $(u, v)$ (right top), and setting
$\alpha=x^{-1}y$ and $\beta=u^{-1}v$, the coloring condition is 
$\alpha^{n} =\beta^{n} =(\alpha\beta)^n $ (in multiplicative notation) with arbitrary choice for $x$ and $u$. 
For every choice of $\alpha, \beta \in \Z_n$ these relations are satisfied. For choices for $x, u, \alpha, \beta$, the number of colorings is $n^4$. 

The cocycle $\phi_i$ takes the value $g$ if and only if the over-arc at a crossing is $(r^j, r^{j+i})$ 
for some $j$, and otherwise takes the value  identity $e$, regardless of colors of the under-arc.
If $(x,y)\times (u, v) = (r^j, r^{j+i}) \times (r^{j'}, r^{j'+i})$, then after the first crossing down, the colorings changes to $(r^{j'}, r^{j'+i}) \times (r^{j+i}, r^{j+2i})$, and inductively, every crossing contributes $g$, in total $g^n$ from $n$ crossings.
There are $n^2$ colorings such that both $(x,y)$ and $(u, v)$ are of this form with $i\neq 0$,  which contribute
$g^n \otimes g^n$ to the invariant from both components, hence we obtain the term $n^2 (g^n \otimes g^n, g^n \otimes g^n)$.

Let us now count the colorings where one component contributes nontrivially and the other does not. Without loss of generality we suppose that the component corresponding to coloring $(x,y)$ contributes nontrivially. There are $n$ choices for $(x,y)=(r^j, r^{j+i})$, for a choice of $i = 1, \ldots , n-1$, as observed above. For $(u,v)$ contributing trivially to the cocycle invariant we need $(u,v) = (r^j,r^{j+k})$ for some $j$ and some $k \neq i$. So we have a total number of choices given by $n^2-n$, since we have $n$ possibilities for $j$ and $k$ each, to which we subtract $n$ colorings corresponding to case $k = i$. These give a summand of the invariant equal to $n(n-1)(e \otimes e, g^n  \otimes g^n)$.
Similarly we obtain $n(n-1) ( g^n  \otimes g^n, e \otimes e)$.
There are $n^4 - [n^2-n(n-1)]= n^4 + n$ colorings such that neither of $(x,y)$ nor $(u,v)$ is of the form
$(r^j, r^{j+i})$
contributing $(e \otimes e, e \otimes e)$. Hence we obtain
\begin{eqnarray*}
\Psi_{\phi_i} (T(2,2n)) &=& 
n^2   (g^n \otimes g^n, g^n \otimes g^n) +  (n^4 + n)(e \otimes e, e \otimes e) \\
&+& n(n-1) [ \   ( g^n  \otimes g^n, e \otimes e) + (e \otimes e, g^n  \otimes g^n)  \ ]  . 
\end{eqnarray*}

Next we consider the linear combinations 
$\phi=\phi_{\vec{a}}=\sum_{i=1}^n a_i  \phi_i$, where $\phi_i$ are 2-cocycles defined above,
for $\vec{a}=(a_i) \in \Z_n$.
If $(x,y)=(r^a, r^{a+j})$ and $(u,v)=(r^b, r^{b+i})$, then this coloring contributes
$ (g^{a_i n}\otimes g^{a_i n}, g^{a_j n}\otimes g^{a_j n})$. There are $n^2$ such colorings. In order to have a component contributing trivially and another component contributing nontrivially, one of the pairs $(x,y)$ or $(u,v)$ has to be of type $(r^k,r^k)$, i.e. monochromatic. So there are $n$ colorings for each pair with this property, while for the other pair we have $n(n-1)$ different colors. The colorins contributing trivially arise when both pairs $(x,y)$ and $(u,v)$ are monochromatic, which amounts to a total number of $n^2$. We have obtained

\begin{eqnarray*}
 \Psi_{\phi_{\vec{a}}}(T(2,2n) ) 
&= & 
n^2 \sum_{i,j} (g^{a_i n}\otimes g^{a_i n}, g^{a_j n}\otimes g^{a_j n}) 
+  n^2 (e \otimes e, e \otimes e) \\
& & + n [ \   \sum_{k}  ( g^{a_k n}   \otimes g^{a_k n} , e \otimes e) +   \sum_{l}  (e \otimes e, g^{a_l n}  \otimes g^{a_l n}   )  \ ].  
\end{eqnarray*}
Observe that summing all the types of colorings as corresponding to the three contributions above we obtain $n^2(n-1)^2 + n^2 + 2n^2(n-1) = n^4$, which is the total number of colorings, as expected. 

 }
 \end{example}

 \begin{example}\label{ex:Dninv}
 {\rm
 We further examine the invariant for the torus link $T(2, 2n)$ 
  with $X=D_n$ for coloring and the $2$-cocycles 
 $\psi_i=\sum_{x \in D_n } [ \sum_{j=0}^{n-1} (  \chi_{(x,r^j,r^{j+i})} +\chi_{ (x, ar^{-j},ar^{-j-i}  ) } )  ] $
  in Lemma~\ref{lem:Dn}. Let the coefficient group  be $A=\Z=\braket{g}$.  
 We continue to use the notation from Example~\ref{ex:Zninv}.

We recall that the coloring conditions from Lemma~\ref{lem:torus2n} are 
$ \alpha^n=\beta^n=(\alpha\beta)^n=1$.
The colorings are determined as follows.
	\begin{itemize}
		\item[1.]
		If both $\alpha$ and $\beta$ are rotations, i.e. 
		$\alpha=r^j$ and $\beta=r^k$ for some $j,k$,  then the coloring conditions are satoisfied for all $\alpha, \beta$ and $x,u$. 
		\item[2.] 
		If  $\alpha$ is a rotation $r^j$ and or $\beta$ is a reflection $r^k a$, then
		$\alpha^n=1$ is satisfied. 
		If $n$ is odd, 
		 then since  $\beta$ has order $2$,  $\beta^n\neq 1$ and  there is   no colorings.
		 		 If $n$ is even, any choice of $\beta$ would satisfy the equation $\beta^n=(\alpha\beta)^n = 1$, since both are reflections and have order 2.
				  The case of $\alpha$ reflection and $\beta$ rotation  is similar. 
		\item[3.] Suppose  both $\alpha, \beta$ are reflections. If $n$ is odd, 
		there is no coloring. If $n$ is even, then any choice would color.
	\end{itemize}
By definition of $\psi_i$, a crossing contributes nontrivially to the cocycle invariant $\Psi$, 
if and only if the over-arc is colored by $(tr^j  ,tr^{j+i} )  $ with $t = 1,a$. 
Suppose $n$ is even. 
If both $(x,y)$ and $(u, v)$ are of this form, then such a coloring contributes $g \in A$ at every crossing as in the preceding example, in total $g^n \otimes g^n$ to the invariant.
There are $n$ choices for $j$ and 2 choices for  $t=1,a$ for  $(x,y), (u,v)=(tr^j  ,tr^{j+i}) $, hence these colorings contribute $4n^2 (g^n \otimes g^n, g^n \otimes g^n)$ to the invariant.
If $(x,y)$ is of this form and $(u,v)$ is not, then the contribution is $(e \otimes e, g^n \otimes g^n)$. In order to have $(u,v)$ not of the form $(tr^j  ,tr^{j+i})$, we have $2n(n-1)$ cases corresponding to colorings of type $(tr^\ell,tr^{\ell+k})$ with $k\neq i$, and $2n^2$ colorings of type $(a^sr^\ell,a^{s+1}r^k)$ where $s$ is taken modulo $2$. So we obtain a total number of colorings of $2n^2(n-1) + 4n^3$. The opposite case is similar.
The remaining cases contribute $(e \otimes e,e \otimes e)$.
Suppose $n$ is odd. Then there is no colorings of cases 2 and 3, and the coefficient
of $(e \otimes e, g^n \otimes g^n)$ and $(g^n \otimes g^n, e \otimes e)$ is
$2n^2(n-1)$ since $\beta$ cannot be reflections in this case, and there are $2n(n-1)$ colorings $(tr^\ell,tr^{\ell+k})$ with $k\neq i$.
The cocycle invariant is therefore
$$
\Psi_{\psi_i}(T(2,2n) )  =\begin{cases}
4n^2  (g^n\otimes g^n, g^n\otimes g^n) +   m (1\otimes 1, 1\otimes 1) \\
+ (4n^2(n-1) + 4n^3) [\ (g^n\otimes g^n,1\otimes 1) + (1\otimes 1, g^n\otimes g^n)\ ] \quad n\ {\rm even}\\
4n^2 (g^n\otimes g^n, g^n\otimes g^n) + m' (1\otimes 1, 1\otimes 1) \\
+ 2n^2(n-1) [ \ (g^n\otimes g^n,1\otimes 1) +  (1\otimes 1, g^n\otimes g^n) \ ] \hspace{18mm}  n\ {\rm odd} .
\end{cases} 
$$
The coefficients $m$, $m'$ are from  the remaining cases.

Next we consider  the linear combinations 
$\phi=\psi_{\vec{a}}=\sum_{i=1}^n a_i  \psi_i$, where $\psi_i$ are 2-cocycles defined above,
for $\vec{a}=(a_i) \in \Z_n$.
If $(x,y)=(tr^a , tr^{a+j} )$ and $(u,v)=(tr^b , tr^{b+i} )$ for $t=1,a$,  then this coloring contributes
$ (g^{a_i n}\otimes g^{a_i n}, g^{a_j n}\otimes g^{a_j n})$. There are $(2n)^2$ such colorings.
Other cases are similar, and we have 
$$
\Psi_{\psi}(T(2,2n) )  =\begin{cases}
4n^2  \sum_{i,j} (g^{a_i n}\otimes g^{a_i n}, g^{a_j n}\otimes g^{a_j n}) 
 +   \ell (1\otimes 1, 1\otimes 1) \\
+ 4n^3[\   \sum_{k}   (g^{a_k n}  \otimes g^{a_k n},1\otimes 1) + 
  \sum_{l } (1\otimes 1, g^{a_l n}\otimes g^{a_l n}  )  \ ] \  \quad n\  {\rm even}\\
4n^2 \sum_{i,j} (g^{a_i n}\otimes g^{a_i n}, g^{a_j n}\otimes g^{a_j n}) 
+ \ell' (1\otimes 1, 1\otimes 1) \\
+ \   n[\sum_{k}  (g^{a_k n}  \otimes g^{a_k n},1\otimes 1) +
 \sum_{l}   (1\otimes 1, g^{a_l n}\otimes g^{a_l n}) \ ]\hspace{9mm}  n\ {\rm odd} .
\end{cases} 
$$

  }
 \end{example}

We show that computations of the cocycle invariant above can be used to derive algebraic consequences in cohomology, by providing lower bounds of the rank of cohomology groups as follows.

\begin{proposition}\label{prop:cyclic/dihedral}
For the cyclic $\Z_n$ and dihedral $D_n$ heaps, we have 
${\rm rank}\ H^2_{\rm NDH} (X, \Z) \geq n-1$, for all $n\geq 2$, where $X=\Z_n, D_n$.
\end{proposition}

\begin{proof}
For $X=\Z_n$, we use the cocycle invariant value of $\phi_{\vec{a}}$ in Example~\ref{ex:Zninv}. 
For any $\vec{a}\neq 0$, the invariant value contains a non-trivial term (not equal to $(e\otimes e, e\otimes e)$). Furthermore, there is a nondegenerate  coloring containing a non-trivial term.
This implies that $\phi_{\vec{a}}$ is non-trivial in $H^2_{\rm NDH}(X, \Z)$ by 
Theorem~\ref{thm:cocyinvariant}. Hence $\{ [ \phi_i ] \} $ is linearly independent, and we obtain the result.
A similar argument applies to the case $X=D_n$ using Example~\ref{ex:Dninv}.
\end{proof}

\noindent
{\it Acknowledgements.} MS was supported in part  by  NSF DMS-1800443. 

\appendix

\section{Proofs for Section~\ref{sec:coset}} \label{app:exactseq}

\noindent
{\it Proof of Example~\ref{ex:Z4}}. 
A cocycles $\phi \in    \widehat{Z}^2_{{\rm N} \{G\}}(X,A) $ satisfies $(*)$ and
$\psi(x, y, y)=0$ for all $x, y \in X$ (nondegenerate condition) and 
$\phi(x, 0, 2)=\phi(x,2,0)=0$, $\phi(x, 1,3)=\phi(x,3,1)=0$ for all $x\in X$ (the localized quotient condition). 
By Lemma~\ref{lem:core}, we have 
\begin{eqnarray*}
& & \phi(x, 0,1)=\phi(x, 1,2)=\phi(x, 2,3) = \phi(x, 3,0 ), \\
& &  \phi(x, 0,3)=\phi(x, 3,2)=\phi(x, 2,1) = \phi(x, 1,0 ) .
 \end{eqnarray*}
 Proposition~\ref{prop:equivariance} implies
$\phi(x,y,z)=\phi(x+2,y+2,z+2)$ for all $x,y,z \in X$. 
 Hence by setting  $\xi_{(x, \pm 1)} =  \sum_{ y  \in X }   \chi_{( x, y, y\pm 1)} $,
  $\phi$ is expressed as 
\begin{eqnarray*}
\phi &=& 
a_0\  [ \xi_{(0, 1)} + \xi_{(2,1)} ] 
\ + \ a_1 \  [ \xi_{(1, 1)} + \xi_{(3,1)} ] \\ 
&+&    b_0\   [ \xi_{(0, -1)} + \xi_{(2,-1)} ] 
\  + \ b_1\   [ \xi_{(1, -1)} + \xi_{(3,-1)} ] . 
 \end{eqnarray*}
 In $(*)$, if $z-y=1=v-u$, then we have
 $$ \phi(x,y,z)-\phi(x+1, y+1, z+1)-\phi(x, u, v) + \phi(x+1, u,v)=0.$$
 For $x=0$ this implies $a_0 - a_1 - a_0 + a_1=0$ for corresponding terms, which is satisfied. 
 Similarly, for all $x$ the above equation is satisfied, as well as the case $z-y=-1=v-u$.
In fact, recall that Proposition~\ref{prop:equivariance} is proved with the case $v-u\in G$, in this case
$v-u=2$ mod 4. The case $z-y=2$ in $(*)$ implies $\phi(x,u,v)=\phi(x+2, u, v)$ for all $u, v$. 
Hence these cases do not introduce new equations among $a_i$ and $b_j$ for $i, j =0,1$.
The two cases (A) $z-y=1$ and $v-u=-1$, (B) $z-y=-1$ and $v-u=1$ remain. 
In Case (A), $(*)$ is written as
$$ \phi(x, y, z) - \phi(x-1, y-1, z-1) - \phi(x, u, v) + \phi(x+1, u, v)=0.$$
For $x=0$, the terms imply, respectively, $a_0 - b_1 - b_0 + b_1=0$, that is, $a_0=b_0$. 
For $x=1$, we obtain  $a_1 - a_0 - b_0 + b_0=0$, so that $a_0=a_1$. 
For $x=2$, we obtain  $a_0 - a_1 - b_0 + b_1=0$, so that $b_0=b_1$. 
Hence we obtain $\psi= a \sum_{x, y \in X} \chi_{(x, y, y \pm 1)}$, and $   \widehat{Z}^2_{{\rm N} \{G\}}(X,A) \cong A$. 
One computes
\begin{eqnarray*}
\delta \chi_{(0)} &=& 
[ \xi_{(0, 1)} +  \xi_{(0, -1)} ] - [  \xi_{(1, -1)} +  \xi_{(3,1)} ] \\
\delta \chi_{(1)} &=&
[  \xi_{(1, 1)} + \xi_{(1, -1)}]-[  \xi_{(0, 1)} +   \xi_{(2,-1)} ]  \\
\delta \chi_{(2)} &=& 
[ \xi_{(2,1)} +  \xi_{(2,-1)} ] -  [  \xi_{(1, 1)} + \xi_{(3,-1)} ] \\
\delta \chi_{(3)} &=& 
[ \xi_{(3,1)}  + \xi_{(3,-1)} ] - [  \xi_{(0, -1)} + \xi_{(2,1)} ] .
\end{eqnarray*}
Hence $\widehat{H}^2_{ {\rm N} \{ G \} } (X, A)\cong A $.

Next we compute $H^2_{ {\rm N} \{ G \} } (X, A)$. 
Generators of $C^2_{ {\rm N} \{ G \} } (X, A)$ are $\chi_{(x, y, z)}$ 
where $\{ y, z \} = \{0, 2\}$ or $\{ 1,3\}$. Let $\eta \in C^2_{ {\rm N} \{ G \} } (X, A)$.  
By Lemma~\ref{lem:core} we have 
$\eta (x, y, z) =\eta (x, z, y)$ for all $x \in X$ 
since $z^{-1}yz=y$ in $X$ for $\{ y, z \} = \{0, 2\}$ or $\{ 1,3\}$.
Recall that $(*)$ reduces to these conditions under $y=u$ and $z=v$,  from the proof of Lemma~\ref{lem:core}.
Set  $\zeta_{(x,y)}:= \chi_{(x, y, y+2)} + \chi_{(x, y+2, y)}$, then 
 $\eta$ can be written as 
$\eta=\sum_{ x \in X, \ y \in \{0,1\}} a_{(x,y)}\zeta_{(x,y)}$.
Since $y \neq z$ and $u\neq v$ by assumption, the equation $(*)$  in $C^2_{ {\rm N} \{ G \} } (X, A)$ 
reduces to the cases 
$y \neq z$, $u\neq v$ 
 and either $y\neq u$ or $z \neq v$. 
 If $(y,z)=(0,2)$, for example, then it follows that 
 $(u,v)=(1,3), (2,0), (3,1)$ under these conditions. In all cases we have 
 $(u,v)=(y+\epsilon, z+\epsilon)$ for $\epsilon =1,2,3$. 
We also have $-y+z\equiv -u+v\equiv 2$ mod $4$, so that $(*)$ reduces to
$$\eta(x,y,z) -\eta(x+2, y+2, z+2) - \eta(x,u,v)  + \eta(x+2, u,v) =0$$
for  $(u,v)=(y+\epsilon, z+\epsilon)$ and  $\epsilon =1,2,3$. 
Since $\eta(x+2, y+2, z+2)=\eta(x+2, z+2, y+2)=\eta(x+2, y, z)$ for $\{ y, z \}= \{0, 2\}$ or $\{ 1,3\}$,
we have 
$\eta(x,y,z) - \eta(x+2, y, z)=  \eta(x,u,v) - \eta(x+2, u,v) $.
This implies that $\eta(x,y,z) - \eta(x+2, y, z)$ is constant over $\{y, z\}=\{ 0, 2\}$ and $\{1,3\}$,
and we obtain that $a_{(x,y)}-a_{(x+2,y)}$ is constant over $y \in \{0,1\}$.
One computes 
$
\delta
\chi_{(x)}=  \sum_{y \in \{ 0, 1 \} } 
[ \zeta_{(x,y)} - \zeta_{(x+2, y)} ] 
$,
where indices are mod 4. Hence  $\delta C^1_{ {\rm N} \{ G \} } (X, A)$ does not introduce new relations among $a_{(x,y)}$s, and we obtain  $H^2_{ {\rm N} \{ G \} } (X, A) \cong A^{\oplus 6}$, since there are 8  varianbles $a_{(x,y)}$ and two relations, $a_{(x,y)}-a_{(x+2,y)}$ being constant over $y \in \{0,1\}$.

From  Proposition~\ref{prop:longexact} we obtain the following exact sequence:
$$ 0 \rightarrow \widehat{H}^2_{{\rm N} \{G\}}(X,A) 
\stackrel{j^*}{\rightarrow} 
H^2_{ {\rm NDH} } (X, A)
\stackrel{i^*}{\rightarrow} 
i^*( H^2_{ {\rm N} \{ G \} } (X, A) ) 
 \rightarrow  0
$$
where $i^*( H^2_{ {\rm N} \{ G \} } (X, A) ) $ is isomorphic to $A^{\oplus r}$ with $r \leq 6$. 
If $A=\Z$ then $i^*( H^2_{ {\rm N} \{ G \} } (X, A) ) $  is free of rank $r\leq 6$ and 
$H^2_{ {\rm NDH} } (X, A)$ is free of rank $\leq 7$. 

\bigskip

\noindent
{\it Proof of Example~\ref{ex:D3}}.
Let $\phi$ be in $\widehat{Z}^2_{{\rm N} \{G,F\}}(X,A)$, 
where we set $A=\Z$ throughout this section, then by definition $\phi$ vanishes on chains $(x,y,z)$ where $y$ and $z$ are in the same $G$- or $F$-coset. 
By direct inspection, we see that for $yG \neq zG$ and $yF \neq zF$ implies $y^{-1}z=ra$ or $r^2a$.
From Proposition~\ref{prop:equivariance} we obtain $\phi(x,y,z) = \phi(xa,ya,za)$ for $y,z$ in different $G$- and $F$-cosets. Similarly, by taking $uF = vF$ (equivariance with respect to $F$) we obtain equations 
$$
\phi(x,y,z) = 
\phi(xr,yr,zr) = \phi(xr^2,yr^2,zr^2).
$$
Together we obtain 
$$\phi(x,y,z) = \phi(xra,yra,zra) = \phi(xr^2 a,yr^2 a,zr^2 a) $$
as well. 
In $(*)$, since $\phi(x,y,z) = \phi(xu^{-1}v , yu^{-1}v ,zu^{-1}v )$ for $u^{-1}v=ra$ and $r^2a$, we obtain
$\phi (xy^{-1}z, u,v)=\phi (x, u, v)$ for $y^{-1}z=ra$ and $r^2a$. By varying $x$, we obtain 
$\phi(x, u, v)=\phi(y, u, v)$ for all $x, y$. 
From this and the equivariance we obtain 
$$\phi(1,y,z) = \phi(1,ya, za) = \phi(1,yr,zr) , $$
 which implies $\phi(x,y,z)$ constant for all $y,z$ in different $G,F$-cosets. 
 Taking $y = 1$ and $z = r$ in  $\phi(1,y,z) = \phi(1,y r, z r)$ 
 we have  $\phi(1,1, r) = \phi(1, r, r^2)$, and $r, r^2$ are in the same $F$-coset, thus $ \phi(1, r, r^2)=0$. 
 Similarly $y = 1$ and $z = a$ in  $\phi(1,y,z) = \phi(1,y r^2, z r^2)$ we have
  $\phi(1,1,a) = \phi(1, r^2, r^2 a)=0$. Hence $\phi=0$ and 
we have $\widehat{H}^2_{{\rm N} \{G,F\}}(X,A) = 0$, from which it follows that $\widehat{H}^2_{{\rm N} \{G\}}(X,A) \cong i^* (\widehat{H}^2_{{\rm N} \{G\}}(X,A)) \leq H^2_{{\rm N} \{G,F \}}(X,A)$.

To compute $i^* (\widehat{H}^2_{{\rm N} \{G\}}(X,A))$, we first characterize $2$-cocycles in $Z^2_{{\rm N} \{G,F \}}(X,A)$. 
Applying Lemma~\ref{lem:core}, we have $\phi(  x,y,z)=\phi(x, z, zy^{-1}z)$. 
By setting $z=yr$ and $yr^2$, we obtain 
$\phi(  x,y,yr)=\phi(x, yr, yr^2) = \phi(x,yr^2,y)$ and $\phi(  x,y,yr^2)=\phi(x, yr^2, yr)= \phi(x,yr,y)$.
Therefore, we can choose one element $y$ from each $F$-coset, namely $y = 1$ and $y = a$ and a $2$-cocycle $\phi$ as  
\begin{eqnarray*}
	\lefteqn{\phi} &=& \sum_{x \in D_3} \{ a_{11}(x)[\chi_{(x,1,r)}+\chi_{(x,r,r^2)}+\chi_{(x,r^2,1)}]\\
	&&\hspace{7mm}  + a_{12}(x)[\chi_{(x,1,r^2)}+\chi_{(x,r,1)}+\chi_{(x,r^2,r)}]\\
	 &&\hspace{7mm}  + a_{21}(x)[\chi_{(x,a,ar)}+\chi_{(x,ar,ar^2)}+\chi_{(x,ar^2,a)}]\\
	&& \hspace{7mm}   + a_{22}(x)[\chi_{(x,a,ar^2)}+\chi_{(x,ar,a)}+\chi_{(x,ar^2,ar)}]
\}.
	\end{eqnarray*}
	
	Observe that for all $y\neq z$ such that $yF = zF$, i.e. $z = yr$ or $yr^2$, we have $y^{-1}z = r, r^2$. Therefore $2$-cocycles in $Z^2_{{\rm N} \{G,F \}}(X,A)$ are characterized by the equations
\begin{eqnarray}
\phi(x,y,yr) - \phi(xr,yr,yr^2) - \phi(x,u,ur) + \phi(xr, u, ur) &=& 0 \label{eq:NGF1} \\
\phi(x,y,yr) - \phi(xr^2,yr^2,y) - \phi(x,u,ur^2) + \phi(xr, u, ur^2) &=& 0\label{eq:NGF2} \\ 
\phi(x,y,yr^2) - \phi(xr,yr,y) - \phi(x,u,ur) + \phi(xr^2, u, ur) &=& 0\label{eq:NGF3} \\ 
\phi(x,y,yr^2) - \phi(xr^2,yr^2,yr) - \phi(x,u,ur^2) + \phi(xr^2, u, ur^2) &=& 0 \label{eq:NGF4}
\end{eqnarray}
along with $G$-equivariance (Proposition \ref{prop:equivariance}).

Equations~(\ref{eq:NGF1}) to~(\ref{eq:NGF4}) written in terms of $a_{ij}(x)$ give 16 equations, as $y,u$ take values $1,a$. A direct inspection shows that they are not all nontrivial and independent. In fact the only three independent equations are those obtained from Equation~(\ref{eq:NGF2}) with $y = u = 1 ,a$ and Equation~(\ref{eq:NGF1}) with $y = 1$ and $u=a$
\begin{eqnarray}
	a_{11}(x) - a_{11}(xr^2) &=& a_{12}(x)-a_{12}(xr)\label{eqn:1'} \\
	a_{21}(x) - a_{21}(xr^2) &=& a_{22}(x)- a_{22}(xr)\label{eqn:2'} \\
	a_{11}(x) - a_{11}(xr^2) &=& a_{22}(x)-a_{22}(xr)\label{eqn:2''}\\
	a_{11}(x) - a_{11}(xr) &=& a_{21}(x) - a_{21}(xr)\label{eqn:3'} 
\end{eqnarray}
Equation~(\ref{eqn:2''}) can be seen to be redundant by considering Equation~(\ref{eqn:3'}) with $xr^2$ instead of $x$, and then applying Equation~(\ref{eqn:2'}). Observe that $G$-equivariance in terms of coefficients $a_{ij}(x)$ is translated as $a_{11}(x) = a_{22}(xa)$ and $a_{12}(x) = a_{21}(xa)$. Therefore by switching $x$ to $xa$ in Equation~(\ref{eqn:1'}) and using $G$-equivariance we obtain Equation~(\ref{eqn:2'}), so that this is redundant as well, and can be eliminated.

Now we determine the image of $\widehat Z^2_{\rm N\{G\}}(X,A)$ under $i^\sharp$.  Let us now suppose that $[\psi]\in \widehat H^2_{\rm N\{G\}}(X,A)$, then $i^\sharp \psi$ is in $\widehat Z^2_{\rm N\{G\}}(X,A)$, which means that $\psi$ evaluated on chains localized at $F$ is determined by coefficients $a_{ij}(x)$ as given above. Since $\psi$ satisfies the $2$-cocycle condition for chains relative to $G$, we obtain new constraints on coefficients $a_{ij}(x)$, from which we determine the image $i^\sharp(\widehat Z^2_{\rm N\{G\}}(X,A))$. 

From $(*)$ with $z = ya$ and $u,v$ in the same $F$-coset, i.e. two cases $v = ur^i$ with $i=1,2$, we get two equations
\begin{eqnarray}
\psi(xa,u,ur) - \psi(x,u,ur) &=& \psi(xr, yr, yr^2a) \label{eqn:psinonlocal1}\\
\psi(xa,u,ur^2) - \psi(x,u,ur^2) &=& \psi(xr^2,yr^2,yra). \label{eqn:psinonlocal2}
\end{eqnarray}
\begin{sloppypar}
\noindent where, as $y$ ranges in $D_3$, $(yr,yr^2a)$ gives all the (unordered) pairs in $A_1 := \{(r,r^2a), (r^2,a), (1,ra)\}$  and $(yr^2,yra)$ gives all the (unordered) pairs in $A_2 := \{(r^2,ra), (1,r^2a), (r,a)\}$, which account for all possible pairs in different $F$ and $G$-cosets. Observe that in both equations it is enough to consider $u = 1, a$, since $\psi$ on chains localized at $F$ is determined by the $F$-coset. So, Equations~(\ref{eqn:psinonlocal1}) and~(\ref{eqn:psinonlocal2}) determine $\psi$ on chains with $y,z$ in different $F,G$-cosets from the value of $\psi$ on chains localized at $F$, where $\psi$ is given by coefficients $a_{ij}(x)$. Observe that Equations~(\ref{eqn:psinonlocal1}) and~(\ref{eqn:psinonlocal2}) do not depend on $u$, from which we obtain that  $\psi(xa,1,1r) - \psi(x,1,1r) = \psi(xa,a,ar) - \psi(x,a,ar)$. 
 \end{sloppypar}
Since 
$\psi$ is $G$-equivariant so that the previous equation becomes $\psi(x,a,ar^2) - \psi(x,1,1r) = \psi(x,1,r^2) - \psi(x,a,ar)$. In terms of coefficients $a_{ij}(x)$ we have obtained 
\begin{eqnarray}
a_{22}(x) - a_{11}(x) &=& a_{12}(x) - a_{21}(x) \label{eqn:6'}.
\end{eqnarray}
Moreover, Equatioins~(\ref{eqn:psinonlocal1}) and~(\ref{eqn:psinonlocal2}) imply that $\psi$ is constant when $(y,z)$ varies in $A_1$ or $A_2$, respectively, and are given by 
\begin{eqnarray}
\psi(x,y,z) &=& a_{22}(xr^2) - a_{11}(xr^2) \quad {\rm when} \ (y,z)\in A_1, \label{eqn:A1} \\
\psi(x,y,z) &=&  a_{11}(xr)- a_{22}(xr) \quad {\rm  when} \  (y,z)\in A_2.\label{eqn:A2} 
\end{eqnarray}

We need to consider $(*)$ for chains that are neither localized at $G$ nor at $F$. 
There are the cases of $(*)$ with at least one of the pairs $(y,z)$ and $(u,v)$ in $A_1$ or $A_2$. There are five cases to consider here:
\begin{itemize}
	\item[1.] 
	$yG = zG$ and $(u,v)\in A_i$, $i = 1,2$,
	\item[2.] 
	$(y,z)\in A_i$, $i = 1,2$ and $uG = vG$,
	\item[3.] 
	$yF = zF$ and $(u,v)\in A_i$, $i = 1,2$,
	\item[4.] 
	$(y,z)\in A_i$, $i = 1,2$ and $uF= vF$, 
	\item[5.] 
	$(y,z)\in A_i$, $i = 1,2$ and $(u,v)\in A_i$, $i = 1,2$. 
\end{itemize}

We start with Case 1. Observe that for any pair 
$(u,v)\in A_1$  
we have that $u^{-1}v = ra$,   
so that we can choose a pair in $A_1$, and the result would not change, a similar consideration holds for $A_2$. Let us take the pair $(u,v) = (1,ra)$ and let $z = ya$. Then $(*)$,
with $G$-equivariance,  gives
$$
-\psi(xr,yr,yar) - \psi(x,1,ra) + \psi(x,a,r) = 0.
$$
Using
Equations~(\ref{eqn:A1}) and (\ref{eqn:A2}), localized, 
we obtain the equation
\begin{eqnarray}
- [ a_{22}(x) - a_{11}(x)] - [a_{22}(xr^2) - a_{11}(xr^2) ] + [a_{11}(xr) - a_{22}(xr) ]= 0 \label{eqn:7'}.
\end{eqnarray}
The case with $(u,v)\in A_2$   
gives the same equation. 

Case 2 is easily seen to be just a restatement of $G$-equivariance of $\psi$, so that it does not give any new equations. 

Case 3, with $z = yr$ and $(u,v) = (1,ra)\in A_1$ gives two equations, depeding on $y$ being $1$ or $a$. When $y =1 $ it follows that $(*)$ is
\begin{eqnarray*}
	\psi(x,1,r) - \psi(xra, ra, r^2a) - \psi(x,1,ra) + \psi(xr,1,ra) = 0.
	\end{eqnarray*}
Since $\psi(x,1,r) = a_{11}(x)$, $\psi(xra, ra, r^2a) = \psi(xr, r, r^2) = a_{11}(xr)$ (having used $G$-equivariance), and using Equation~(\ref{eqn:A1}) also $\psi(x, 1, ra) = a_{22}(xr) - a_{11}(xr)$, $\psi(xr, 1, ra) = a_{22}(x) - a_{11}(x)$, we obtain
\begin{eqnarray}
-a_{11}(xr) - a_{22}(xr^2) + a_{11}(xr^2) + a_{22}(x) = 0.  \label{eqn:8'}
\end{eqnarray} 

This equation is obtained from Equation~(\ref{eqn:2''}) with $xr$ instead of $x$. So it is a redundant equation. When $y = a$, we similarly obtain the equation
\begin{eqnarray*}
	a_{21}(x) - a_{21}(xr) - a_{22}(xr) + a_{11}(xr) + a_{22}(x) - a_{11}(x) = 0,
	\end{eqnarray*}
which is seen to be Equation~(\ref{eqn:8'}) by applying Equation~(\ref{eqn:3'}). Similarly, the cases with $z= yr^2$ and $(u,v) = (1,ra)\in A_1$ give again Equation~(\ref{eqn:8'}). Direct computation also shows that when $z = yr^i$ and $(u,v) = (1,r^2a)\in A_2$ we obtain four equations that are equivalent to 
$$
a_{11}(x) - a_{11}(xr) + a_{22}(xr) - a_{22}(xr^2) = 0,
$$
which is seen to be redundant from Equation~(\ref{eqn:3'}) by substituting $a_{21}(x) - a_{21}(xr)$ with $a_{22}(xr^2) - a_{22}(xr)$ as obtained from Equation~(\ref{eqn:2'}) with $xr$ instead of $x$. So no further symmetries are obtained from the eight equations corresponding to case 3. 

 Let us now consider Case 4, i.e. $(*)$ with $(y,z)\in A_1$ and $v = ur$.  Then we obtain
$$
\psi(x,1,ra) - \psi(xr,r,rar) - \psi(x,u,ur) + \psi(xra,u,ur) = 0,
$$
which can be rewritten, using $G$-equivariance on the last term as
$$
\psi(x,1,ra) - \psi(xr,r,a) - \psi(x,u,ur) + \psi(xr,ua,uar^2) = 0.
$$
Taking $u=1$ and writing the previous equation in terms of $a_{ij}(x)$ we get
\begin{eqnarray*}
2[a_{22}(xr^2) - a_{11}(xr^2)] - a_{11}(x) + a_{22}(xr) = 0. 
\end{eqnarray*}
This equation is obtained from summing together Equations~(\ref{eqn:7'}) and~(\ref{eqn:8'}), so that this is redundant. 
The other cases are seen similarly to give redundant equations.

Finally, consider Case  5, when both $(y,z)$ and $(u,v)$ are either in $A_1$ or in $A_2$. When both $(y,z)$ and $(u,v)$ are in $A_1$, say they are both $(1,ra)$, we have 
$$
\psi(x,1,ra) - \psi(xra,ra,1) - \psi(x,1,ra) + \psi(xra,1,ra) = 0,
$$
which upon using $G$-equivariance gives $\psi(xr,r,a) = \psi(xr,a,r)$. This it already known, as $\psi$ is constant on $(x,y,z)$ as the pair $(y,z)$ ranges in $A_2$. The case with $(y,z)$ and $(u,v)$ in $A_2$ gives similarly a redundant equation. When we take $(y,z)\in A_1$, say $(1,ra)$, and $(u,v)\in A_2$, say $(1,r^2a)$ we obtain 
$$
\psi(x,1,ra) - \psi(xr^2a, r^2a,r^2) - \psi(x,1,r^2a) + \psi(xra, 1, r^2a) = 0,
$$
and therefore we get
$$
a_{22}(xr^2) - a_{11}(xr^2) - a_{11}(xr) + a_{22}(xr) + a_{22}(x) - a_{11}(x) = 0.
$$
But this is Equation~(\ref{eqn:7'}), so no new symmetries are introduced. Similarly we see that when we take $(y,z)\in A_2$ and $(u,v)\in A_1$ the same equation is obtained. 

We have found that the equations determining $\psi$ are Equations~(\ref{eqn:1'}) and (\ref{eqn:3'}), Equation~(\ref{eqn:6'}) and Equation~(\ref{eqn:7'}).
As previously stated, recall that G-equivariance in terms of $a_{ij}(x)$ gives symmetries: $a_{11}(x) = a_{22}(xa)$ and $a_{12}(x) = a_{21}(xa)$. Therefore we can consider $x = 1,r,r^2$, as the equations with $x = a, ar, ar^2$ are obtained from the former ones by $G$-equivariance. Now we analyze the cases corresponding to $x= 1, r , r^2$. Observe that Equation~(\ref{eqn:1'}) gives three equations, of which one is redundant. Similarly Equation~(\ref{eqn:3'}) gives two non-redundant equations, while (\ref{eqn:6'}) splits in three cases, and (\ref{eqn:7'}) is fixed as $x = 1, r, r^2$. We have a total of 8 equations, which by direct computation are seen to reduce the number of free parameters $a_{ij}(x)$ from 12 to 4. Specifically, we have equations
\begin{eqnarray*}
a_{11}(1) &=& a_{11}(r^2) + a_{12}(1) - a_{12}(r) \\
a_{11}(r) &=& a_{11}(r^2) -  a_{12}(1) + 3a_{21}(r) - 2a_{12}(r)\\
a_{21}(1) &=& a_{12}(r) -2a_{21}(r) + 2a_{12}(1) \\
a_{21}(r^2) &=& -2a_{21}(r) + a_{12}(1) + 2a_{12}(r)\\
a_{22}(1) &=& a_{11}(r^2) + 2a_{21}(r) - 2a_{12}(r)\\
a_{22}(r) &=& a_{11}(r^2) - a_{12}(1) + 2a_{21}(r) - a_{12}(r) \\
a_{22}(r^2) &=& a_{11}(r^2) - a_{21}(r) + a_{12}(1) \\
a_{12}(r^2) &=& -3a_{21}(r) + 2a_{12}(1) + 2a_{12}(r)
\end{eqnarray*}
where $a_{11}(r^2)$, $a_{12}(r)$, $a_{12}(r^2)$ and $a_{21}(r)$ are free parameters in $A$, and $a_{ij}(x)$ with $x = a, ar, ar^2$ are obtained by $G$-equivariance. 
These equations are obtained from the above 4 equations for $x=1$ and $x=r$ by solving them for the terms that appear in the LHS. 


We have found: $\widehat Z^2_{\rm N\{G\}}(X,A) \cong A^{\oplus 4}$. 
A $2$-cocycle in $\hat{Z}^2_{\rm N\{G\}}(X,A)$ is written explicitly as 
 \begin{eqnarray*}
 \psi &=& \sum_x [a_{11}(x)\chi_{11}(x) + a_{12}(x)\chi_{12}(x)+ a_{21}(x) \chi_{21}(x)+ a_{22}(x) \chi_{22}(x)]\\
 && + \sum_x [a_{22}(xr^2)-a_{11}(xr^2)]\xi_1(x) + \sum_x [a_{11}(xr)-a_{22}(xr)]\xi_2(x),
 \end{eqnarray*}
where, for $i,j=1, 2$,  we have set $\chi_{ij}= \sum_{y\in F} \chi_{(x,a^{j-1}y,a^{j-1}yr^i)}$, $\xi_k(x) = \sum_{y \in X} \chi_{(x,yr^k,yar^k)}$ with $\xi_1(x) = \xi_2(xa)$, and the coefficients $a_{ij}(x)$ are determined by four free parameters as written above. To obtain the coboundaries, we determine $\delta \chi_{(x)}$ for $x\in X$. This is given by
$$
\delta \chi_{(r^i)} = \sum_{j= 1,2} [\chi_{j1}(r^i)-\chi_{j1}(r^{i-j}) + \xi_j(r^i) - \xi_j(r^{i+j}a)],
$$ 
while $\delta \chi_{(ar^i)}$ can be obtained from the previous one by substituting $ar^i$ in place of $r^i$. Since a $1$-cochain $f$ obtaind as a sum of characteristics functions $\delta \chi_{(x)}$ has to be $G$-equivariant, we see that $\delta \chi_{(r^i)}$ and $\delta \chi_{(ar^i)}$ 
are related by $G$-equivariance, and we can restrict ourselves to considering $\delta \chi_{(r^i)}$, for $i = 0, 1, 2$. Moreover, observe that $\delta \chi_{(1)} + \delta \chi_{(r)} = - \delta \chi_{(r^2)}$, so that the image $\delta \widehat Z^1_{\rm N\{G\}}(X,A)$ 
is generated by $\delta \chi_{(1)}$ and $\delta \chi_{(r)}$. We see therefore that in the quotient we obtain 
${\rm rank}\ \widehat H^2_{\rm N\{G\}}(X,A)  \leq 2$. In fact in Proposition~\ref{prop:cyclic/dihedral} it is proved the lower bound ${\rm rank}\ \widehat H^2_{\rm N\{G\}}(X,A)  \geq 2$, so the second cohomology group has rank $2$.

To compute $Z^2_{\rm N\{G\}}(X,A)$, let us consider a $2$-cocycle $\phi$ localized at $G$ and $(*)$ with $y,z$ and $u,v$ in the same $G$-coset. Upon changing variable $y$ we obtain the symmetry $\phi(x,y,ya) = \phi(x,ya,y)$. From this we can rewrite the $2$-cocycle condition on chains localized at $G$ as
$$
\phi(x,y,ya) - \phi(xa,y,ya) = \phi(x,u,ua) - \phi(xa,u,ua),
$$
for all $y,u$.
For each $x\in X$, letting $y,u$ vary in $X$, we obtain three equations:
\begin{eqnarray*}
\phi(x,1,a) - \phi(x,r,ra) &=& \phi(xa,1,a) - \phi(xa,r,ra)\\
\phi(x,1,a) - \phi(x, r^2,r^2a) &=& \phi(xa,1,a) - \phi(xa,r^2,r^2a)\\
\phi(x,r,ra) - \phi(x,r^2,r^2a) &=& \phi(xa,r,ra) - \phi(xa,r^2,r^2a)
\end{eqnarray*}
from which, considering $a(x) := \phi(x,1,a)$, $b(x):= \phi(x,r,ra)$, $c(x):= \phi(x,r^2,r^2a)$ and $d(x):= \phi(xa,1,a)$ as free parameters in $A$, we also determine the remaining two values of $\phi$ evaluated on $(xa,r,ra)$ and $(xa,r^2,r^2a)$. It is enough to consider only $x= 1,r, r^2$, since the previous three equations would automatically fix the values of $\phi$ for $xa = a, ra, r^2a$. Therefore we have four free parameters $a(x)$, $b(x)$, $c(x)$ and $d(x)$ with $x = 1,r,r^2$. This corresponds to a $12$-dimensional $Z^2_{\rm N\{G\}}(X,A)$ whose elements are written explicitly as
\begin{eqnarray*}
\phi &=& \sum_{x=1,r,r^2} [a(x)\chi_{(x,1,a)} + b(x)\chi_{(x,r,ra)} + c(x)\chi_{(x,r^2,r^2a)} + d(x)\chi_{(xa,1,a)}\\
&&+ (-a(x)+b(x)+d(x))\chi_{(xa,r,ra)} + (-a(x)+c(x)+ d(x))\chi_{(xa,r^2,r^2a)}].
\end{eqnarray*}
For $f: X\rightarrow A$, on chains localized at $G$, we have $\delta f(x,y,ya) = f(x) - f(xa)$, so that $\delta f$ does not depend on $y$, hence it satisfies $\delta f(x,y,ya) = \delta f(x,ya,y)$, 
and the symmetries 
 $\delta f(x,y,ya) = - \delta f(xa,y,ya)$.
We can write $\delta f$ as 
$$
\delta f = \sum_{x = 1,r,r^2} \alpha[\chi'_{(x,1,a)} + \chi'_{(x,r,ra)} + \chi'_{(x,r^2,r^2a)}- \chi'_{(xa,1,a)} - \chi'_{(xa,r,ra)} - \chi'_{(xa,r^2,r^2a)}],
$$
where $\chi'_{(x,y,ya)} := \chi_{(x,y,ya)} + \chi_{(x,ya,y)}$. 
This implies ${\rm rank}\ H^2_{\rm N\{G\}}(X,A) \leq 9$. 
Therefore from the exact sequence
$$
0 \longrightarrow  \widehat H^2_{\rm N\{G\}}(X,A)  \longrightarrow H^2_{\rm NDH}(X,A)  \longrightarrow i^*(H^2_{\rm NDH}(X,A)) \longrightarrow 0, 
$$
and ${\rm rank}  \widehat H^2_{\rm N\{G\}}(X,A) = 2$, 
${\rm rank}\ i^*(H^2_{\rm NDH}(X,A)) \leq {\rm rank} H^2_{\rm N\{G\}}(X,A) \leq 9$, 
we obtain an upper bound for the rank of cohomology group: ${\rm rank}\ H^2_{\rm N\{G\}}(X,A) \leq 11$.

\section{Proofs for Section~\ref{sec:fund}} \label{app:fund}

\noindent
{\it Proof for Lemma~\ref{lem:torus2n}}.
	We proceed by induction on $n$ positive, noting that the formulas hold for $n=1$ and $n = 2, 3$, since 
	\begin{eqnarray*}
		\sigma_1 ( (x,y)\times (u,v) )&=& (u,v)\times (xu^{-1}v, yu^{-1}v) = (u,v)\times (x\beta, y\beta),\\
		\sigma_1^2 ( (x,y)\times (u,v) )&=& (x\beta,y\beta)\times (u\beta^{-1}\alpha\beta,v\beta^{-1}\alpha\beta),\\
		\sigma_1^3 (  (x,y)\times (u,v) )&=& (u\beta^{-1}\alpha\beta,v\beta^{-1}\alpha\beta)\times (x\alpha^{-1}\beta\alpha\beta, y\alpha^{-1}\beta\alpha\beta).
	\end{eqnarray*}
	Let us now assume that the equations hold for some $n> 3$. We distinguish two cases, based on the parity of $n$. Suppose first that $n = 2k$ for some $k$. Since $\sigma_1^{n+1} = \sigma_1 \sigma_1^n = \sigma_1 \sigma_1^{2k}$, by induction hypothesis we obtain
	\begin{eqnarray*}
		\sigma_1^{n+1}( (x,y)\times (u,v) ) 
		&=& \sigma_1 ( 
	(x\alpha^{-k}(\alpha\beta)^{k}, y\alpha^{-k}(\alpha\beta)^{k})\times (u\beta^{-k}(\alpha\beta)^k,v\beta^{-k}(\alpha\beta)^k ) 
		) \\
		&=& (u\beta^{-k}(\alpha\beta)^k,v\beta^{-k}(\alpha\beta)^k) 
		 \times ( 	x\alpha^{-k}(\alpha\beta)^{k}\cdot (\alpha\beta)^{-k}\beta^ku^{-1}v\beta^{-k}(\alpha\beta)^k,
 \\ & & \hspace{4.7cm}
		 y\alpha^{-k}(\alpha\beta)^{k}\cdot (\alpha\beta)^{-k}\beta^ku^{-1}v\beta^{-k}(\alpha\beta)^k).
\end{eqnarray*}
	Here $\alpha^{-k}(\alpha\beta)^{k}\cdot (\alpha\beta)^{-k}\beta^ku^{-1}v\beta^{-k}(\alpha\beta)^k = \alpha^{-k}\beta (\alpha\beta)^k = \alpha^{-(k+1)}  (\alpha\beta)^{k+1}$, so this case is complete. Similarly, if $n = 2k+1$, we have
	\begin{eqnarray*}
\lefteqn{ 		( \sigma_1\sigma_1^n )( (x,y)\times (u,v) )
		=
			(x\alpha^{-(k+1)} (\alpha\beta)^{k+1} ,y\alpha^{-(k+1)} (\alpha\beta)^{k+1}) }
		\\
		& \times & (u\beta^{-k}(\alpha\beta)^k(\alpha\beta)^{-k}\beta^{-1}\alpha^kx^{-1}y \alpha^{-k}\beta(\alpha\beta)^k,
			v \beta^{-k}(\alpha\beta)^k(\alpha\beta)^{-k}\beta^{-1}\alpha^kx^{-1}y \alpha^{-k}\beta(\alpha\beta)^k),
\end{eqnarray*}
	where, from $\beta^{-k}(\alpha\beta)^k(\alpha\beta)^{-k}\beta^{-1}\alpha^kx^{-1}y\alpha^{-k}\beta(\alpha\beta)^k = \beta^{k+1}(\alpha\beta)^{k+1}$, this case holds true as well, and the proof is complete. The case of $n$ negative is done similiarly and it is therefore omitted.

\bigskip

\noindent
{\it Proof of Proposition~\ref{thm:pretzelnoframe}}.
	 Assign  generators $(x_i,y_i)$, $i=1 , \ldots, r$, to the upper left (double) arcs of $b_i$.
	 Since the upper right  arcs of $b_i$ are identified with the upper left  arcs of $b_{i+1}$ in the opposite order (double arcs are parallel and do not cross), the upper right arc of  $b_i$ is assigned 
	 $(y_{i+1}, x_{i+1})$. 
	  We apply Lemma~\ref{lem:torus2n} to $b_i = \sigma_1^{2k_i+1}$.
	Since $(x, y) \times (u, v) = (x_i , y_i) \times  (y_{i+1}, x_{i+1})$ in the lemma, we set 
	$\alpha=x^{-1}y=\alpha_i := x_i^{-1}y_i$. 
	Then $\beta=u^{-1}v=y_{i+1}^{-1} x_{i+1}=\alpha_{i+1}^{-1}$.

	If $n_i=2k_i$  are all even, then   the lower arcs of $b_i$ are colored by 
$$
	(x_i \alpha_i^{-k_i}(\alpha_i \alpha_{i+1}^{-1} )^{k_i}, 
	y_i \alpha_i ^{-k_i}(\alpha_i \alpha_{i+1}^{-1} )^{k_i}  )
	\times 
	( y_{i+1} \alpha_{i+1}^{-k_i}  (\alpha_i \alpha_{i+1}^{-1} )^{k_i} ,
	x_{i+1} \alpha_{i+1}^{-k_i}  (\alpha_i \alpha_{i+1}^{-1} )^{k_i}   ) .
$$	
 Identifications of the lower arcs (the right bottom arcs of $b_{i}$ and the left bottom arcs of $b_{i+1}$ in opposite order) 
 leads to the following relations for all $i\in \Z_r$:
 \begin{eqnarray*}
  y_{i+1} \alpha_{i+1}^{-k_i}  (\alpha_i \alpha_{i+1}^{-1} )^{k_i}
  &=&
y_{i+1} \alpha_{i+1}^{-k_{i+1}}(\alpha_{i+1} \alpha_{i+2}^{-1} )^{k_{i+1}}  , \label{even1} \\
x_{i+1} \alpha_{i+1}^{-k_i}  (\alpha_i \alpha_{i+1}^{-1} )^{k_i} 
&=&
x_{i+1}  \alpha_{i+1}^{-k_{i+1}}(\alpha_{i+1} \alpha_{i+2}^{-1} )^{k_{i+1}} .
  \end{eqnarray*}
These give the same relator $\Theta_i$.


\begin{thebibliography}{0}
	\bib{BH}{book}{
			title={Cogroups and co-rings in categories of associative rings},
			author={Bergman, George M},
			author={Hausknecht, Adam O},
			number={45},
			year={1996},
			publisher={American Mathematical Soc.}
		}
	


	\bib{CJKLS}{article}{
		title={Quandle cohomology and state-sum invariants of knotted curves and surfaces},
		author={Carter, J},
		author={ Jelsovsky, Daniel},
		author={Kamada, Seiichi},
		author={Langford, Laurel},
		author={Saito, Masahico},
		journal={Transactions of the American Mathematical Society},
		volume={355},
		number={10},
		pages={3947--3989},
		year={2003}
	}
	\bib{EGM}{article}{
		author={Elhamdadi, Mohamed},
		author={Green, Matthew},
		author={Makhlouf, Abdenacer},
		title={Ternary distributive structures and quandles},
		journal={Kyungpook Math. J.},
		volume={56},
		date={2016},
		number={1},
		pages={1--27},
		issn={1225-6951},
		review={\MR{3609661}},
		doi={10.5666/KMJ.2016.56.1.1},
	}
\bib{ESZheap}{article}{
	title={Heap and ternary self-distributive cohomology},
	author={Elhamdadi, Mohamed},
	author={Saito, Masahico},
	author={Zappala, Emanuele},
	journal={Communications in Algebra},
	URL={https://doi.org/10.1080/00927872.2020.1871484},
	year={2021}
}
\bib{ESZ}{article}{
	author = {Elhamdadi, Mohamed},
	author={Saito, Masahico},
	author={Zappala, Emanuele},
	title = {Higher arity self-distributive operations in Cascades and their cohomology},
	journal = {Journal of Algebra and Its Applications},
	volume = {},
	number = {},
	year = {},
	doi = {10.1142/S0219498821501164},	
	URL = { 
		https://doi.org/10.1142/S0219498821501164		
	},
	eprint = { 
		https://doi.org/10.1142/S0219498821501164	
	},
}

	\bib{FR}{article}{
	   author={Fenn, Roger},
	   author={Rourke, Colin},
	   title={Racks and Links in Codimension Two},
	   journal={Journal of Knot Theory and Its Ramifications},
	   volume={1},
	   number={4},
	   date={1992},
	   pages={343--406},
	}	

		\bib{Green}{article}{
			author={Green, Matthew J.},
			title={Generalizations of Quandles and their cohomologies},
			journal={PhD dissertation, University of South Florida},
			volume={},
			date={2018},
			number={},
			pages={},
			issn={},
			review={},
			doi={},
		}


%

\bib{LitherlandNelson}{article}{
	title={The Betti numbers of some finite racks},
	author={Litherland, RA},
	author={Nelson, Sam},
	journal={Journal of Pure and Applied Algebra},
	volume={178},
	number={2},
	pages={187--202},
	year={2003},
	publisher={Elsevier}
}


\bib{NOO}{article}{
	title={Local biquandles and Niebrzydowski's tribracket theory},
	author={Nelson, Sam},
	author={Oshiro, Kanako},
	author= {Oyamaguchi, Natsumi},
	journal={Topology and its Applications},
	volume={258},
	pages={474--512},
	year={2019},
	publisher={Elsevier}
}

\bib{Nie1}{article}{
	title={On some ternary operations in knot theory},
	author={Niebrzydowski, Maciej},
	journal={Fundamenta Mathematicae},
	year={2014},
	volume={225},
	pages={259--276}
}

\bib{Nie2}{article}{
	title={Homology of ternary algebras yielding invariants of knots and knotted surfaces},
	author={Niebrzydowski, Maciej},
	journal={Algebraic \& Geometric Topology},
	volume={20},
	number={5},
	pages={2337--2372},
	year={2020},
	publisher={Mathematical Sciences Publishers}
}
	
	
	\bib{Vin}{article}{
		title={Groups defined by periodic paired relations},
		author={Vinberg, Ernest Borisovich},
		journal={Sbornik: Mathematics},
		volume={188},
		number={9},
		pages={1269},
		year={1997},
		publisher={IOP Publishing}
	}
\bib{SZsfceribbon}{article}{
	title={Fundamental heaps for  surface ribbons  and  cocycle invariants},
	author={Saito, Masahico},
	author={Zappala, Emanuele},
	journal={Preprint, arXiv:2109.07569},
	year={2021}
}

\bib{EZ}{article}{
	title={Non-Associative Algebraic Structures in Knot Theory},
	author={Zappala, Emanuele},
	journal={Ph.D. dissertation, University of South Florida},
	year={2020}
}

\end{thebibliography}
\end{document}